\RequirePackage{fix-cm}
\documentclass{svjour3}                     
\smartqed  
\usepackage{graphicx,cite,geometry}
\usepackage{amssymb,amsmath,mathrsfs,scalerel,latexsym,upgreek}
\begin{document}

\title{Positive Radial Solutions for an Iterative System of Nonlinear Elliptic Equations in an Annulus}

\titlerunning{Positive Radial Solutions for an Iterative System of Nonlinear Elliptic Eq.}        

\author{Mahammad Khuddush${}^\star$ \and K. Rajendra Prasad}

\authorrunning{M. Khuddush, K. R. Prasad} 

\institute{   Department of Applied Mathematics\\
              College of Science and Technology\\
              Andhra University, Visakhapatnam, 530003, India\\
              ${}^\star$Corresponding author: khuddush89@gmail.com         
}

\date{Received: date / Accepted: date}

\maketitle

\begin{abstract}
This paper deals with the existence of positive radial solutions to the iterative system of nonlinear elliptic equations of the form 
$$
\begin{aligned}
	\triangle{\mathtt{u}_{{\dot{\iota}} }}-\frac{(\mathtt{N}-2)^2r_0^{2\mathtt{N}-2}}{\vert x\vert^{2\mathtt{N}-2}}\mathtt{u}_{\dot{\iota}} +\ell(\vert x\vert)\mathtt{g}_{{\dot{\iota}} }(\mathtt{u}_{{\dot{\iota}} +1})=0,~\mathtt{R}_1<\vert x\vert<\mathtt{R}_2,
\end{aligned}
$$
where ${\dot{\iota}} \in\{1,2,3,\cdot\cdot\cdot,\mathtt{n}\},$ $  \mathtt{u}_1=  \mathtt{u}_{\mathtt{n}+1},$ $\triangle{\mathtt{u}}=\mathtt{div}(\triangledown  \mathtt{u}),$ $\mathtt{N}>2,$ $\ell=\prod_{i=1}^{m}\ell_i,$ each $\ell_i:(r_0,+\infty)\to(0,+\infty)$ is continuous, $r^{\mathtt{N}-1}\ell$ is integrable, and  $\mathtt{g}_{\dot{\iota}} :[0,+\infty)\to\mathbb{R}$ is continuous,  by an application of various fixed point theorems in a Banach space. Further, we also establish uniqueness of solution to the addressed system by using Rus's theorem in a complete metric space. 
\end{abstract}

\section{Introduction}
The semilinear elliptic equation of the form
\begin{equation}\label{eq01}
	\triangle \mathtt{u}+\mathtt{g}(\vert x\vert)\mathtt{u}+\mathtt{h}(\vert x\vert)\mathtt{u}^\mathtt{p}=0
\end{equation}
arise in various fields of pure and applied mathematics such as Riemannian geometry, nuclear physics, astrophysics and so on. For more details of the background of \eqref{eq01}, see \cite{ni, yana}. Study of nonlinear elliptic system of equations,
\begin{equation}\label{eq11}\left.
	\begin{aligned}
		&\triangle{  \mathtt{u}_{{\dot{\iota}} }}+ \mathtt{g}_{{\dot{\iota}} }(  \mathtt{u}_{{\dot{\iota}} +1})=0~~\text{in}~~\Omega,\\
		&\hskip1cm  \mathtt{u}_{{\dot{\iota}} }=0~~\text{on}~~\partial\Omega,\\
	\end{aligned}\right\}
\end{equation}
where ${\dot{\iota}} \in\{1,2,3,\cdot\cdot\cdot,{\mathtt{n}}\},$ $  \mathtt{u}_1=  \mathtt{u}_{\mathtt{n}+1},$ and $\Omega$ is a bounded domain in $\mathbb{R}^N$, has an important applications in population dynamics, combustion theory and chemical reactor theory. The recent literature for the existence, multiplicity and uniqueness of positive solutions for \eqref{eq11}, see \cite{dal, hai1, ali1, hai2, hai3, ali2} and references therein.

In \cite{mei}, Mei studied the structure of radial solutions to the quasilinear elliptic problem of the form
$$
\begin{aligned}
	&\triangle_{\mathtt{p}}\mathtt{u}-\frac{\uplambda}{\mathtt{u}^\mathtt{q}}=0~\text{in}~\Omega,\\
	&\hskip0.4cm0<\mathtt{u}<1~\text{in}~\Omega,\\
	&\hskip0.6cm\mathtt{u}=1~\text{on}~\partial\Omega,
\end{aligned}
$$
where $1<\mathtt{p}\le\mathtt{N}\,(\mathtt{N}\ge2),\mathtt{q}>0,\uplambda>0$ and $\Omega=\{x\in\mathbb{R}^\mathtt{N}:\vert x\vert<1\}.$ In \cite{li1}, Li studied the existence of positive radial solutions of the elliptic equation with nonlinear gradient term
$$
\begin{aligned}
	&-\triangle \mathtt{u}=\mathtt{g}\big(\vert x\vert, \mathtt{u},\frac{x}{\vert x\vert}\cdot\nabla \mathtt{u}\big)=0~~\text{in}~~\Omega_a^b,\\
	&\hskip1.6cm \mathtt{u}=0~~\text{on}~~\partial\Omega_a^b,
\end{aligned}
$$
by using fixed point index theory in cones. In \cite{chro}, Chrouda and Hassine established the uniqueness of positive radial solutions to the following Dirichlet boundary value problem for the semilinear elliptic equation in an annulus,
$$
\begin{aligned}
\triangle\mathtt{u}=\mathtt{g}(\mathtt{u})&~~\text{on}~~\Omega=\{x\in\mathbb{R}^\mathtt{d}:a<\vert x\vert<b\},\\
&\mathtt{u}=0~~\text{on}~~\mathtt{u}\in\partial\Omega,
\end{aligned}
$$
for any dimension $\mathtt{d}\ge1.$ In \cite{kaji}, R. Kajikiya and E. Ko established the existence of positive radial solutions for a semipositone elliptic equation of the form,
$$
\begin{aligned}
	&\triangle \mathtt{u}+\uplambda\mathtt{g}(\mathtt{u})=0~~\text{in}~~\Omega,\\
	&\hskip0.7cm \mathtt{u}=0~~\text{on}~~\partial\Omega,
\end{aligned}
$$
where $\Omega$ is a ball or an annulus in $\mathbb{R}^\mathtt{N}.$ Recently, Son and Wang \cite{son} established positive radial solutions to the nonlinear elliptic systems,
$$
\begin{aligned}
	&\triangle{  \mathtt{u}_{{\dot{\iota}} }}+\uplambda\mathtt{K}_{\dot{\iota}} (\vert x\vert)\mathtt{g}_{{\dot{\iota}} }(  \mathtt{u}_{{\dot{\iota}} +1})=0~\text{in}~\Omega_\mathtt{E},\\
	&\hskip1.25cm  \mathtt{u}_{{\dot{\iota}} }=0~\text{on}~\vert x\vert=r_0,\\
	&\hskip1.05cm  \mathtt{u}_{{\dot{\iota}} }\to0~\text{as}~\vert x\vert\to+\infty,
\end{aligned}
$$
where ${\dot{\iota}} \in\{1,2,3,\cdot\cdot\cdot,{\mathtt{n}}\},$ $  \mathtt{u}_1=  \mathtt{u}_{\mathtt{n}+1},$ $\uplambda>0,$ $N>2,$ $r_0>0,$ and $\Omega_\mathtt{E}$ is an exterior of a ball. Motivated by the above works, in this paper we study the existence of infinitely many positive radial solutions for the following iterative system of nonlinear elliptic equations in an annulus,
\begin{equation}\label{eq12}
	\triangle{\mathtt{u}_{{\dot{\iota}} }}-\frac{(\mathtt{N}-2)^2r_0^{2\mathtt{N}-2}}{\vert x\vert^{2\mathtt{N}-2}}\mathtt{u}_{\dot{\iota}} +\ell(\vert x\vert)\mathtt{g}_{{\dot{\iota}} }(\mathtt{u}_{{\dot{\iota}} +1})=0,~\mathtt{R}_1<\vert x\vert<\mathtt{R}_2,
\end{equation}	
with one of the following sets of boundary conditions:
\begin{equation}\label{eq131}\left.
	\begin{aligned}
		&\hskip0.9cm \mathtt{u}_{{\dot{\iota}} }=0~~\text{on}~~\vert x\vert=\mathtt{R}_1~\text{and}~\vert x\vert=\mathtt{R}_2,\\
		&\mathtt{u}_{{\dot{\iota}} }=0~~\text{on}~~\vert x\vert=\mathtt{R}_1~\text{and}~\frac{\partial \mathtt{u}_{\dot{\iota}} }{\partial r}=0~\text{on}~\vert x\vert=\mathtt{R}_2,\\
		&\frac{\partial \mathtt{u}_{\dot{\iota}} }{\partial r}=0~~\text{on}~~\vert x\vert=\mathtt{R}_1~\text{and}~\mathtt{u}_{{\dot{\iota}} }=0~\text{on}~\vert x\vert=\mathtt{R}_2,
	\end{aligned}\right\}
\end{equation}
where ${\dot{\iota}} \in\{1,2,3,\cdot\cdot\cdot,\mathtt{n}\},$ $  \mathtt{u}_1=  \mathtt{u}_{\mathtt{n}+1},$ $\triangle\mathtt{u}=\mathtt{div}(\triangledown\mathtt{u}),$ $N>2,$ $\ell=\prod_{i=1}^{m}\ell_i,$ each $\ell_i:(\mathtt{R}_1,\mathtt{R}_2)\to(0,+\infty)$ is continuous, $r^{\mathtt{N}-1}\ell$ is integrable, by an application of various fixed point theorems in a Banach space. Further, we also study existence of unique solution by using Rus's theorem in a complete metric space. 

The study of positive radial solutions to \eqref{eq12} reduces to the study of positive solutions to the following iterative system of two-point boundary value problems,
\begin{equation}\label{eq13}
	\begin{aligned}
		\mathtt{u}''_{{\dot{\iota}} }(\mathtt{s})-r_0^2\mathtt{u}_{\dot{\iota}} (\mathtt{s})+\ell(\mathtt{s})\mathtt{g}_{{\dot{\iota}} }(\mathtt{u}_{{\dot{\iota}} +1}(\mathtt{s}))=0,~0<\mathtt{s}<1,
	\end{aligned}
\end{equation}
where ${\dot{\iota}} \in\{1,2,3,\cdot\cdot\cdot,\mathtt{n}\},$ $  \mathtt{u}_1=  \mathtt{u}_{\mathtt{n}+1},$ $r_0>0$ and $\ell(\mathtt{s})=\frac{r_0^2}{(\mathtt{N}-2)^2}\mathtt{s}^\frac{2(\mathtt{N}-1)}{2-\mathtt{N}}\prod_{i=1}^{m}\ell_i(\mathtt{s}),$ $\ell_i(\mathtt{s})=\ell_i(r_0\mathtt{s}^\frac{1}{2-\mathtt{N}})$ by a Kelvin type transformation through the change of variables $r=\vert x\vert$ and $\mathtt{s}=\left(\frac{r}{r_0}\right)^{2-\mathtt{N}}.$ The detailed explanation of the transformation from the equation \eqref{eq14} to \eqref{eq13} see \cite{ali3, lan, lee}. By suitable choices of nonnegative real numbers $\upalpha, \upbeta, \upgamma$ and $\updelta,$ the set of boundary conditions \eqref{eq13} reduces to
\begin{equation}\label{eq130}\left\{
	\begin{aligned}
		&\upalpha\mathtt{u}_{\dot{\iota}} (0)-\upbeta\mathtt{u}_{\dot{\iota}} '(0)=0,\\
		&\,\upgamma\mathtt{u}_{\dot{\iota}} (1)+\updelta\mathtt{u}_{\dot{\iota}} '(1)=0,
	\end{aligned}\right.
\end{equation}
we assume that the following conditions hold throughout the paper:
\begin{itemize}
	\item [$(\mathcal{J}_1)$] $\mathtt{g}_{\dot{\iota}} :[0,+\infty)\to[0,+\infty)$ is continuous.
	\item [$(\mathcal{J}_2)$] $\ell_i\in L^{\mathtt{p}_i}[0,1], 1\le\mathtt{p}_i\le+\infty$ for $1\le i\le n.$
	\item [$(\mathcal{J}_3)$] There exists $\ell_i^\star>0$ such that $\ell_i^\star<\ell_i(\mathtt{s})<\infty$ a.e. on $[0, 1].$ 
\end{itemize}
The rest of the paper is organized in the following fashion. In Section 2, we convert the boundary value problem \eqref{eq13}--\eqref{eq130} into equivalent integral equation which involves the kernel. Also, we estimate bounds for the kernel which are useful in our main results. In Section 3, we develop a criteria for the existence of atleast one positive radial solution by applying Krasnoselskii’s cone fixed point theorem in a Banach space. In Section 4, we derive necessary conditions for the existence of atleast two positive radial solution by an application of Avery-Henderson cone fixed point theorem in a Banach space. In Section 5, we establish the existence of atleast three positive radial solution by utilizing Legget-William cone fixed point theorem in a Banach space. Further, we also study uniqueness of solution in the final section.
\section{Kernel and Its Bounds}
\noindent In order to study BVP \eqref{eq13}, we first consider the corresponding linear boundary value problem,
\begin{equation}\label{eq14}
		-\mathtt{u}_1''(\mathtt{s})+r_0^2\mathtt{u}_1(\mathtt{s})=u(\mathtt{s}),~0<\mathtt{s}<1,
\end{equation}
\begin{equation}\label{eq141}\left\{
	\begin{aligned}
		&\upalpha\mathtt{u}_1(0)-\upbeta\mathtt{u}_1'(0)=0,\\
		&\,\upgamma\mathtt{u}_1(1)+\updelta\mathtt{u}_1'(1)=0,
	\end{aligned}\right.
\end{equation}
where $u\in \mathcal{C}[0, 1]$ is a given function. 
\begin{lemma}\label{l21} Let $\varrho=r_0^2(\upalpha\updelta+\upbeta\upgamma)\cosh(r_0)+r_0(\upalpha\upgamma+\upbeta\updelta r_0^2)\sinh(r_0).$ For every $u\in \mathcal{C}[0, 1],$ the linear boundary value problem \eqref{eq14}--\eqref{eq141} has a unique solution 
	\begin{equation}\label{eq33}
		\mathtt{u}_1(\mathtt{s})=\int_0^1\Xi_{r_0}(\mathtt{s}, \mathtt{t})u(\mathtt{t})\mathtt{d}\mathtt{t},
	\end{equation}
	where
	$$
	\begin{aligned}
		\Xi_{r_0}(\mathtt{s},\mathtt{t})=\frac{1}{\varrho}\left\{\begin{array}{ll}\big(\upalpha\sinh(r_0\mathtt{s})+\upbeta r_0\cosh(r_0\mathtt{s})\big)\big(\upgamma\sinh(r_0(1-\mathtt{t}))+\updelta r_0\cosh(r_0(1-\mathtt{t}))\big), \hskip0.4cm 0\le \mathtt{s}\le \mathtt{t}\le 1,\vspace{1.2mm}\\	
			\big(\upalpha\sinh(r_0\mathtt{t})+\upbeta r_0\cosh(r_0\mathtt{t})\big)\big(\upgamma\sinh(r_0(1-\mathtt{s}))+\updelta r_0\cosh(r_0(1-\mathtt{s}))\big),
			\hskip0.4cm  0\le \mathtt{t}\le\mathtt{s}\le 1.
		\end{array}
		\right.
	\end{aligned}
	$$
\end{lemma}
\begin{lemma}\label{l22} Let $\displaystyle\wp=\max\left\{\frac{\upbeta r_0}{\upalpha\sinh(r_0)+\upbeta r_0\cosh(r_0)},\frac{\updelta r_0}{\upgamma\sinh(r_0)+\updelta r_0\cosh(r_0)}\right\}.$
	The kernel $\Xi_{r_0}(\mathtt{s},\mathtt{t})$ has the following properties:
	\begin{itemize}
		\item[(i)] $\Xi_{r_0}(\mathtt{s},\mathtt{t})$ is nonnegative and continuous on $[0, 1] \times [0, 1],$
		\item[(ii)] $\Xi_{r_0}(\mathtt{s},\mathtt{t})\leq \Xi_{r_0}(\mathtt{t},\mathtt{t})$ for $\mathtt{s},\mathtt{t}\in[0, 1],$
		\item[(iii)] $\wp \Xi_{r_0}(\mathtt{t},\mathtt{t})\leq \Xi_{r_0}(\mathtt{s},\mathtt{t})$ for $\mathtt{s}, \mathtt{t}\in [0, 1].$
	\end{itemize} 
\end{lemma}
\begin{proof}
	From the definition of kernel $\Xi_{r_0}(\mathtt{s}, \mathtt{t}),$ it is clear that $(i)$ holds. To prove $(ii),$ consider
	$$
	\begin{aligned}
	\frac{\Xi_{r_0}(\mathtt{s},\mathtt{t})}{\Xi_{r_0}(\mathtt{t},\mathtt{t})}=&
	\left\{\begin{array}{ll}\displaystyle\frac{\upalpha\sinh(r_0\mathtt{s})+\upbeta r_0\cosh(r_0\mathtt{s})}{\upalpha\sinh(r_0\mathtt{t})+\upbeta r_0\cosh(r_0\mathtt{t})}, \hskip2.045cm 0 \le \mathtt{s}\le \mathtt{t}\le 1,\vspace{1.2mm}\\	
		\displaystyle\frac{\upgamma\sinh(r_0(1-\mathtt{s}))+\updelta r_0\cosh(r_0(1-\mathtt{s}))}{\upgamma\sinh(r_0(1-\mathtt{t}))+\updelta r_0\cosh(r_0(1-\mathtt{t}))},
		\hskip0.4cm  0 \le \mathtt{t}\le \mathtt{s}\le 1,
	\end{array}
	\right.\\
	\le&
	\left\{\begin{array}{ll}\displaystyle1, \hskip6.05cm 0 \le \mathtt{s}\le \mathtt{t}\le 1,\vspace{1.2mm}\\	
	1,\hskip6.05cm  0 \le \mathtt{t}\le \mathtt{s}\le 1,
	\end{array}
	\right.
	\end{aligned}
	$$
	which proves $(ii).$ Finally for $(iii),$ consider
	$$
	\begin{aligned}
		\frac{\Xi_{r_0}(\mathtt{s},\mathtt{t})}{\Xi_{r_0}(\mathtt{t},\mathtt{t})}=&
		\left\{\begin{array}{ll}\displaystyle\frac{\upalpha\sinh(r_0\mathtt{s})+\upbeta r_0\cosh(r_0\mathtt{s})}{\upalpha\sinh(r_0\mathtt{t})+\upbeta r_0\cosh(r_0\mathtt{t})}, \hskip2.45cm 0 \le \mathtt{s}\le \mathtt{t}\le 1,\vspace{1.2mm}\\	
			\displaystyle\frac{\upgamma\sinh(r_0(1-\mathtt{s}))+\updelta r_0\cosh(r_0(1-\mathtt{s}))}{\upgamma\sinh(r_0(1-\mathtt{t}))+\updelta r_0\cosh(r_0(1-\mathtt{t}))},
			\hskip0.8cm  0 \le \mathtt{t}\le \mathtt{s}\le 1,
		\end{array}
		\right.\\
		\ge&
		\left\{\begin{array}{ll}\displaystyle\frac{\upbeta r_0}{\upalpha\sinh(r_0)+\upbeta r_0\cosh(r_0)}, \hskip2.82cm 0 \le \mathtt{s}\le \mathtt{t}\le 1,\vspace{1.2mm}\\	
			\displaystyle\frac{\updelta r_0}{\upgamma\sinh(r_0)+\updelta r_0\cosh(r_0)},
			\hskip2.87cm  0 \le \mathtt{t}\le \mathtt{s}\le 1.
		\end{array}
		\right.
	\end{aligned}
	$$ 	
	This completes the proof.     \qed
\end{proof}	

From Lemma \ref{l21}, we note that an $\mathtt{n}$-tuple $( \mathtt{u}_1, \mathtt{u}_2,\cdot\cdot\cdot, \mathtt{u}_\mathtt{n})$ is a solution of the boundary value problem \eqref{eq13}--\eqref{eq130} if and only, if
$$
\begin{aligned}
	\mathtt{u}_1(\mathtt{s})=&\,\int_{0}^{1}\Xi_{r_0}(\mathtt{s}, \mathtt{t}_1)\ell(\mathtt{t}_1)\mathtt{g}_1\Bigg[\int_{0}^{1}\Xi_{r_0}(\mathtt{t}_1, \mathtt{t}_2)\ell(\mathtt{t}_2)\mathtt{g}_2\Bigg[\int_{0}^{1}\Xi_{r_0}(\mathtt{t}_2, \mathtt{t}_3)\ell(\mathtt{t}_3)\mathtt{g}_4 \cdots\\
	&\mathtt{g}_{\mathtt{n}-1}\Bigg[\int_{0}^{1}\Xi_{r_0}(\mathtt{t}_{\mathtt{n}-1}, \mathtt{t}_\mathtt{n})\ell(\mathtt{t}_\mathtt{n})\mathtt{g}_\mathtt{n}\big( \mathtt{u}_1(\mathtt{t}_\mathtt{n})\big)\mathtt{d}\mathtt{t}_\mathtt{n}\Bigg]\cdot\cdot\cdot \Bigg]\mathtt{d}\mathtt{t}_3\Bigg]\mathtt{d}\mathtt{t}_2\Bigg]\mathtt{d}\mathtt{t}_1.
\end{aligned}
$$
In general,
$$
\begin{aligned}
\mathtt{u}_{\dot{\iota}} (\mathtt{s})=&\,\int_{0}^{1}\Xi_{r_0}(\mathtt{s}, s)\ell(s)\mathtt{g}_{\dot{\iota}} \big( \mathtt{u}_{{\dot{\iota}} +1}(s)\big)ds,~{\dot{\iota}} =1,2,3,\cdot\cdot\cdot,{\mathtt{n}},\\
\mathtt{u}_1(\mathtt{s})=&\, \mathtt{u}_{\mathtt{n}+1}(\mathtt{s}).
\end{aligned}
$$

We denote the Banach space $\mathtt{C}((0, 1),\mathbb{R})$ by $\mathtt{B}$ with the norm $\Vert  \mathtt{u}\Vert=\displaystyle\max_{\mathtt{s}\in[0,1]}\vert  \mathtt{u}(\mathtt{s})\vert.$  The cone $\mathtt{E} \subset \mathtt{B}$ is defined by
$$\mathtt{E}=\Big\{  \mathtt{u}\in \mathtt{B} :  \mathtt{u}(\mathtt{s})\ge0~\text{on}~[0,1]~\text{and}~
\min_{\mathtt{s}\in {[0,\,1]}} \mathtt{u}(\mathtt{s})\geq \wp\Vert  \mathtt{u}\Vert\Big\}.$$
For any $ \mathtt{u}_1\in \mathtt{E},$ define an operator $\aleph :\mathtt{E}\rightarrow \mathtt{B}$ by

\begin{align}
	(\aleph  \mathtt{u}_1)(\mathtt{s})=&\,\int_{0}^{1}\Xi_{r_0}(\mathtt{s}, \mathtt{t}_1)\ell(\mathtt{t}_1)\mathtt{g}_1\Bigg[\int_{0}^{1}\Xi_{r_0}(\mathtt{t}_1, \mathtt{t}_2)\ell(\mathtt{t}_2)\mathtt{g}_2\Bigg[\int_{0}^{1}\Xi_{r_0}(\mathtt{t}_2, \mathtt{t}_3)\ell(\mathtt{t}_3)\mathtt{g}_4 \cdots\nonumber\\
	&\mathtt{g}_{\mathtt{n}-1}\Bigg[\int_{0}^{1}\Xi_{r_0}(\mathtt{t}_{\mathtt{n}-1}, \mathtt{t}_\mathtt{n})\ell(\mathtt{t}_\mathtt{n})\mathtt{g}_\mathtt{n}\big( \mathtt{u}_1(\mathtt{t}_\mathtt{n})\big)\mathtt{d}\mathtt{t}_\mathtt{n}\Bigg]\cdot\cdot\cdot \Bigg]\mathtt{d}\mathtt{t}_3\Bigg]\mathtt{d}\mathtt{t}_2\Bigg]\mathtt{d}\mathtt{t}_1.\label{eqmain}
\end{align}

\begin{lemma}\label{l24}
	$\aleph (\mathtt{E})\subset \mathtt{E}$ and $\aleph :\mathtt{E}\rightarrow \mathtt{E}$ is completely continuous.
\end{lemma}
\begin{proof} 
	 Since $\mathtt{g}_{\dot{\iota}} (\mathtt{u}_{{\dot{\iota}} +1}(\mathtt{s}))$ is nonnegative for $\mathtt{s}\in [0, 1],$ $ \mathtt{u}_1 \in \mathtt{E}.$	Since $\Xi_{r_0}(\mathtt{s}, s),$ is nonnegative for all $\mathtt{s}, \mathtt{t} \in [0, 1],$ it follows that
	$\aleph ( \mathtt{u}_1(\mathtt{s}))\geq 0$ for all $\mathtt{s}\in [0, 1],\, \mathtt{u}_1 \in \mathtt{E}.$ Now, by Lemma \ref{l21} and \ref{l22}, we have
	$$
	\begin{aligned}
		\min_{\mathtt{s}\in[0,1]}(\aleph\mathtt{u}_1)(\mathtt{s})
		&=\,\min_{\mathtt{s}\in[0,1]}\Bigg\{\int_{0}^{1}\Xi_{r_0}(\mathtt{s}, \mathtt{t}_1)\ell(\mathtt{t}_1)\mathtt{g}_1\Bigg[\int_{0}^{1}\Xi_{r_0}(\mathtt{t}_1, \mathtt{t}_2)\ell(\mathtt{t}_2)\mathtt{g}_2\Bigg[\int_{0}^{1}\Xi_{r_0}(\mathtt{t}_2, \mathtt{t}_3)\ell(\mathtt{t}_3)\mathtt{g}_4 \cdots\\
		&\hskip1.4cm\mathtt{g}_{\mathtt{n}-1}\Bigg[\int_{0}^{1}\Xi_{r_0}(\mathtt{t}_{\mathtt{n}-1}, \mathtt{t}_\mathtt{n})\ell(\mathtt{t}_\mathtt{n})\mathtt{g}_\mathtt{n}\big( \mathtt{u}_1(\mathtt{t}_\mathtt{n})\big)\mathtt{d}\mathtt{t}_\mathtt{n}\Bigg]\cdot\cdot\cdot \Bigg]\mathtt{d}\mathtt{t}_3\Bigg]\mathtt{d}\mathtt{t}_2\Bigg]\mathtt{d}\mathtt{t}_1\Bigg\}\\
		&\ge\wp\Bigg\{\int_{0}^{1}\Xi_{r_0}(\mathtt{t}_1, \mathtt{t}_1)\ell(\mathtt{t}_1)\mathtt{g}_1\Bigg[\int_{0}^{1}\Xi_{r_0}(\mathtt{t}_1, \mathtt{t}_2)\ell(\mathtt{t}_2)\mathtt{g}_2\Bigg[\int_{0}^{1}\Xi_{r_0}(\mathtt{t}_2, \mathtt{t}_3)\ell(\mathtt{t}_3)\mathtt{g}_4 \cdots\\
		&\hskip1.4cm\mathtt{g}_{\mathtt{n}-1}\Bigg[\int_{0}^{1}\Xi_{r_0}(\mathtt{t}_{\mathtt{n}-1}, \mathtt{t}_\mathtt{n})\ell(\mathtt{t}_\mathtt{n})\mathtt{g}_\mathtt{n}\big( \mathtt{u}_1(\mathtt{t}_\mathtt{n})\big)\mathtt{d}\mathtt{t}_\mathtt{n}\Bigg]\cdot\cdot\cdot \Bigg]\mathtt{d}\mathtt{t}_3\Bigg]\mathtt{d}\mathtt{t}_2\Bigg]\mathtt{d}\mathtt{t}_1\Bigg\}\\
		&\ge\wp\Bigg\{\int_{0}^{1}\Xi_{r_0}(\mathtt{s}, \mathtt{t}_1)\ell(\mathtt{t}_1)\mathtt{g}_1\Bigg[\int_{0}^{1}\Xi_{r_0}(\mathtt{t}_1, \mathtt{t}_2)\ell(\mathtt{t}_2)\mathtt{g}_2\Bigg[\int_{0}^{1}\Xi_{r_0}(\mathtt{t}_2, \mathtt{t}_3)\ell(\mathtt{t}_3)\mathtt{g}_4 \cdots\\
		&\hskip1.4cm\mathtt{g}_{\mathtt{n}-1}\Bigg[\int_{0}^{1}\Xi_{r_0}(\mathtt{t}_{\mathtt{n}-1}, \mathtt{t}_\mathtt{n})\ell(\mathtt{t}_\mathtt{n})\mathtt{g}_\mathtt{n}\big( \mathtt{u}_1(\mathtt{t}_\mathtt{n})\big)\mathtt{d}\mathtt{t}_\mathtt{n}\Bigg]\cdot\cdot\cdot \Bigg]\mathtt{d}\mathtt{t}_3\Bigg]\mathtt{d}\mathtt{t}_2\Bigg]\mathtt{d}\mathtt{t}_1\Bigg\}\\
		&\ge\wp\,\max_{\mathtt{s}\in[0,1]}\,\vert \aleph   \mathtt{u}_1(\mathtt{s})\vert.
	\end{aligned}
	$$
	Thus $\aleph (\mathtt{E})\subset \mathtt{E}.$ Therefore, by the means of Arzela-Ascoli theorem, the operator $\aleph $ is completely continuous. \qed
\end{proof}
\section{Existence of at Least  One Positive Radial Solution}
In this section, we establish the existence of at least one positive radial solution for the system \eqref{eq13}--\eqref{eq130} by an application of following theorems. 
\begin{theorem}\cite{guo}\label{t41}
	Let $\mathtt{E}$ be a cone in a Banach space $\mathtt{B}$ and let $\mathtt{G},\, \mathtt{F}$ be open sets with $0\in\mathtt{G}, \overline{\mathtt{G}}\subset \mathtt{F}.$ Let $\aleph :\mathtt{E}\cap(\overline{\mathtt{F}}\backslash \mathtt{G})\rightarrow \mathtt{E}$ be a completely continuous operator such that 
	\begin{itemize}
		\item[(i)] $\Vert \aleph  \mathtt{u}\Vert\ge\Vert  \mathtt{u}\Vert,\, \mathtt{u}\in \mathtt{E}\cap\partial\mathtt{G},$ and $\Vert \aleph  \mathtt{u}\Vert\le\Vert  \mathtt{u}\Vert,\,  \mathtt{u}\in \mathtt{E}\cap\partial\mathtt{F},$ or
		\item[(ii)] $\Vert \aleph  \mathtt{u}\Vert\le\Vert  \mathtt{u}\Vert,\,  \mathtt{u}\in \mathtt{E}\cap\partial\mathtt{G},$ and $\Vert \aleph  \mathtt{u}\Vert\ge\Vert  \mathtt{u}\Vert,\,  \mathtt{u}\in \mathtt{E}\cap\partial\mathtt{F}.$ 
	\end{itemize}
	Then $\aleph $ has a fixed point in $\mathtt{E}\cap(\overline{\mathtt{F}}\backslash \mathtt{E}).$
\end{theorem}
\begin{theorem}(H\"older's)\label{holder}
	Let $\mathtt{f}\in L^{\mathtt{p}_i}[0,1]$ with $\mathtt{p}_i>1,$ for $i=1, 2,\cdots, n$ and $\displaystyle\sum_{i=1}^{m}\frac{1}{\mathtt{p}_i}=1.$ Then $\prod_{i=1}^{m}\mathtt{f}_i\in L^{1}[0, 1]$ and 
	$\left\Vert \prod_{i=1}^{m}\mathtt{f}_i\right\Vert_1\le\prod_{i=1}^{m}\Vert \mathtt{f}_i\Vert_{\mathtt{p}_i}.$
	Further, if $\mathtt{f}\in L^1[0,1]$ and $\mathtt{g}\in L^\infty[0,1].$ Then $\mathtt{fg} \in L^1[0,1]$ and $\Vert \mathtt{fg}\Vert_1\le\Vert \mathtt{f}\Vert_1\Vert \mathtt{g}\Vert_\infty.$
\end{theorem}
Consider the following three possible cases for $\ell_i\in L^{\mathtt{p}_i}[0,1]:$ 
$$\sum_{i=1}^{m}\frac{1}{\mathtt{p}_i}<1,~\sum_{i=1}^{m}\frac{1}{\mathtt{p}_i}=1,~\sum_{i=1}^{m}\frac{1}{\mathtt{p}_i}>1.$$

Firstly, we seek positive radial solutions for the case $\displaystyle\sum_{i=1}^{m}\frac{1}{\mathtt{p}_i}<1.$ 
\begin{theorem}\label{t43}
	Suppose $(\mathcal{J}_1)$--$(\mathcal{J}_3)$ hold. Further, assume that there exist two positive constants $a_2>a_1>0$ such that
	\begin{itemize}
		\item[$(\mathcal{J}_4)$] $\mathtt{g}_{\dot{\iota}} ( \mathtt{u}(\mathtt{s}))\le \mathfrak{Q}_2a_2$ for all $0\le\mathtt{s}\le1,\, 0\le \mathtt{u} \le a_2,$ 	where $\displaystyle\mathfrak{Q}_2=\left[\frac{ r_0^2}{(\mathtt{N}-2)^2}\Vert\widehat{\Xi}_{r_0}\Vert_\mathtt{q}\prod_{i=1}^{m}\Vert\ell_i\Vert_{\mathtt{p}_i}\right]^{-1}$ and $\widehat{\Xi}_{r_0}(\mathtt{t})=\Xi_{r_0}(\mathtt{t}, \mathtt{t})\mathtt{t}^\frac{2(\mathtt{N}-1)}{2-\mathtt{N}}.$
		\item[$(\mathcal{J}_5)$] $\mathtt{g}_{\dot{\iota}} ( \mathtt{u}(\mathtt{s}))\ge \mathfrak{Q}_1a_1$ for all $0\le\mathtt{s}\le1,\, 0\le \mathtt{u} \le a_1,$ where $\displaystyle\mathfrak{Q}_1=\left[\frac{\wp r_0^2}{(\mathtt{N}-2)^2}\prod_{i=1}^{m}\ell_i^\star\int_{0}^{1}\Xi_{r_0}(\mathtt{t}, \mathtt{t})\mathtt{t}^\frac{2(\mathtt{N}-1)}{2-\mathtt{N}} \mathtt{d}\mathtt{t}\right]^{-1}.$	
	\end{itemize}
	Then iterative system \eqref{eq13}--\eqref{eq130} has atleast one positive radial solution $(\mathtt{u}_1,\mathtt{u}_2,\cdot\cdot\cdot,\mathtt{u}_\mathtt{n})$ such that $a_1\le\Vert\mathtt{u}_{\dot{\iota}} \Vert\le a_2,\,{\dot{\iota}} =1,2,\cdot\cdot\cdot,\mathtt{n}.$
\end{theorem}
\begin{proof}	
	Let $\mathtt{G}=\{\mathtt{u}\in\mathtt{B}:\Vert\mathtt{u}\Vert<a_2\}.$ For $\mathtt{u}_1\in\partial\mathtt{G},$ we have $0\le\mathtt{u}\le a_2$ for all $\mathtt{s}\in[0,1].$ It follows from $(\mathcal{J}_4)$ that for $\mathtt{t}_{\mathtt{n}-1}\in[0,1],$
	$$
	\begin{aligned}
		\int_{0}^{1}\Xi_{r_0}(\mathtt{t}_{\mathtt{n}-1}, \mathtt{t}_\mathtt{n})\ell(\mathtt{t}_\mathtt{n})\mathtt{g}_\mathtt{n}\big( \mathtt{u}_1(\mathtt{t}_\mathtt{n})\big)\mathtt{d}\mathtt{t}_\mathtt{n}
		&\le \int_{0}^{1}\Xi_{r_0}(\mathtt{t}_{\mathtt{n}}, \mathtt{t}_\mathtt{n})\ell(\mathtt{t}_\mathtt{n})\mathtt{g}_\mathtt{n}\big( \mathtt{u}_1(\mathtt{t}_\mathtt{n})\big)\mathtt{d}\mathtt{t}_\mathtt{n}\\
		&\le \mathfrak{Q}_2a_2\int_{0}^{1}\Xi_{r_0}(\mathtt{t}_{\mathtt{n}}, \mathtt{t}_\mathtt{n})\ell(\mathtt{t}_\mathtt{n})\mathtt{d}\mathtt{t}_\mathtt{n}\\
		&\le \mathfrak{Q}_2a_2\frac{r_0^2}{(\mathtt{N}-2)^2}\int_{0}^{1}\Xi_{r_0}(\mathtt{t}_{\mathtt{n}}, \mathtt{t}_\mathtt{n})\mathtt{t}^\frac{2(\mathtt{N}-1)}{2-\mathtt{N}}_\mathtt{n}\prod_{i=1}^{m}\ell_i(\mathtt{t}_\mathtt{n})\mathtt{d}\mathtt{t}_\mathtt{n}.
	\end{aligned}
	$$
	There exists a $\mathtt{q}>1$ such that $\displaystyle\sum_{i=1}^{m}\frac{1}{\mathtt{p}_i}+\frac{1}{\mathtt{q}}=1.$ By the first part of Theorem \ref{holder}, we have
	$$
	\begin{aligned}
		\int_{0}^{1}\Xi_{r_0}(\mathtt{t}_{\mathtt{n}-1}, \mathtt{t}_\mathtt{n})\ell(\mathtt{t}_\mathtt{n})\mathtt{g}_\mathtt{n}\big( \mathtt{u}_1(\mathtt{t}_\mathtt{n})\big)\mathtt{d}\mathtt{t}_\mathtt{n}&\le \mathfrak{Q}_2a_2\frac{ r_0^2}{(\mathtt{N}-2)^2}\Vert\widehat{\Xi}_{r_0}\Vert_\mathtt{q}\prod_{i=1}^{m}\Vert\ell_i\Vert_{\mathtt{p}_i}\\
		&\le a_2.
	\end{aligned}
	$$
	It follows in similar manner for $0<\mathtt{t}_{\mathtt{n}-2}<1,$
	$$
	\begin{aligned}
		\int_{0}^{1}\Xi_{r_0}(\mathtt{t}_{\mathtt{n}-2}, \mathtt{t}_{\mathtt{n}-1})\ell(\mathtt{t}_{\mathtt{n}-1})\mathtt{g}_{\mathtt{n}-1}&\Bigg[ \int_{0}^{1}\Xi_{r_0}(\mathtt{t}_{\mathtt{n}-1}, \mathtt{t}_\mathtt{n})\ell(\mathtt{t}_\mathtt{n})\mathtt{g}_\mathtt{n}\big( \mathtt{u}_1(\mathtt{t}_\mathtt{n})\big)\mathtt{d}\mathtt{t}_\mathtt{n}\Bigg]\mathtt{d}\mathtt{t}_{\mathtt{n}-1}\\
		&\le\int_{0}^{1}\Xi_{r_0}(\mathtt{t}_{\mathtt{n}-1}, \mathtt{t}_{\mathtt{n}-1})\ell(\mathtt{t}_{\mathtt{n}-1})\mathtt{g}_{\mathtt{n}-1}(a_2)\mathtt{d}\mathtt{t}_{\mathtt{n}-1}\\
		&\le\mathfrak{Q}_2a_2\int_{0}^{1}\Xi_{r_0}(\mathtt{t}_{\mathtt{n}-1}, \mathtt{t}_{\mathtt{n}-1})\ell(\mathtt{t}_{\mathtt{n}-1})\mathtt{d}\mathtt{t}_{\mathtt{n}-1}\\
		&\le \mathfrak{Q}_2a_2\frac{ r_0^2}{(\mathtt{N}-2)^2}\Vert\widehat{\Xi}_{r_0}\Vert_\mathtt{q}\prod_{i=1}^{m}\Vert\ell_i\Vert_{\mathtt{p}_i}\\
		&\le a_2.
	\end{aligned}
	$$
	Continuing with this bootstrapping argument, we reach
	$$
	\begin{aligned}
		(\aleph  \mathtt{u}_1)(t)=&\,  \int_{0}^{1}\Xi_{r_0}(\mathtt{s}, \mathtt{t}_1)\ell(\mathtt{t}_1)\mathtt{g}_1\Bigg[ \int_{0}^{1}\Xi_{r_0}(\mathtt{t}_1, \mathtt{t}_2)\ell(\mathtt{t}_2)\mathtt{g}_2\Bigg[ \int_{0}^{1}\Xi_{r_0}(\mathtt{t}_2, \mathtt{t}_3)\ell(\mathtt{t}_3)\mathtt{g}_4 \cdots\\
		&\hskip1cm\mathtt{g}_{\mathtt{n}-1}\Bigg[ \int_{0}^{1}\Xi_{r_0}(\mathtt{t}_{\mathtt{n}-1}, \mathtt{t}_\mathtt{n})\ell(\mathtt{t}_\mathtt{n})\mathtt{g}_\mathtt{n}\big( \mathtt{u}_1(\mathtt{t}_\mathtt{n})\big)\mathtt{d}\mathtt{t}_\mathtt{n}\Bigg]\cdot\cdot\cdot \Bigg]\mathtt{d}\mathtt{t}_3\Bigg]\mathtt{d}\mathtt{t}_2\Bigg]\mathtt{d}\mathtt{t}_1\\
		\le&\,a_2.
	\end{aligned}
	$$
	Since $\mathtt{G}=\Vert \mathtt{u}_1\Vert$ for $ \mathtt{u}_1\in \mathtt{E}\cap\partial{\mathtt{G}},$ we get
	\begin{equation}\label{eq41}
		\Vert\aleph  \mathtt{u}_1\Vert\le\Vert  \mathtt{u}_1\Vert.
	\end{equation}
	Next, let $\mathtt{F}=\{\mathtt{u}\in\mathtt{B}:\Vert\mathtt{u}\Vert<a_1\}.$ For $\mathtt{u}_1\in\partial\mathtt{F},$ we have $0\le\mathtt{u}\le a_1$ for all $\mathtt{s}\in[0,1].$ It follows from $(\mathcal{J}_5)$ that for $\mathtt{t}_{\mathtt{n}-1}\in[0,1],$
	$$
	\begin{aligned}
	\int_{0}^{1}\Xi_{r_0}(\mathtt{t}_{\mathtt{n}-1}, \mathtt{t}_\mathtt{n})\ell(\mathtt{t}_\mathtt{n})\mathtt{g}_\mathtt{n}\big( \mathtt{u}_1(\mathtt{t}_\mathtt{n})\big)\mathtt{d}\mathtt{t}_\mathtt{n}
	&\ge \wp\int_{0}^{1}\Xi_{r_0}(\mathtt{t}_{\mathtt{n}}, \mathtt{t}_\mathtt{n})\ell(\mathtt{t}_\mathtt{n})\mathtt{g}_\mathtt{n}\big( \mathtt{u}_1(\mathtt{t}_\mathtt{n})\big)\mathtt{d}\mathtt{t}_\mathtt{n}\\
	&\ge \wp\mathfrak{Q}_1a_1\int_{0}^{1}\Xi_{r_0}(\mathtt{t}_{\mathtt{n}}, \mathtt{t}_\mathtt{n})\ell(\mathtt{t}_\mathtt{n})\mathtt{d}\mathtt{t}_\mathtt{n}\\
	&\ge \mathfrak{Q}_1a_1\frac{\wp r_0^2}{(\mathtt{N}-2)^2}\int_{0}^{1}\Xi_{r_0}(\mathtt{t}_{\mathtt{n}}, \mathtt{t}_\mathtt{n})\mathtt{t}^\frac{2(\mathtt{N}-1)}{2-\mathtt{N}}_\mathtt{n}\prod_{i=1}^{m}\ell_i(\mathtt{t}_\mathtt{n})\mathtt{d}\mathtt{t}_\mathtt{n}\\
	&\ge \mathfrak{Q}_1a_1\frac{\wp r_0^2}{(\mathtt{N}-2)^2}\prod_{i=1}^{m}\ell_i^\star\int_{0}^{1}\Xi_{r_0}(\mathtt{t}_{\mathtt{n}}, \mathtt{t}_\mathtt{n})\mathtt{t}^\frac{2(\mathtt{N}-1)}{2-\mathtt{N}}_\mathtt{n} \mathtt{d}\mathtt{t}_\mathtt{n}\\
	&\ge a_1.
	\end{aligned}
	$$
	It follows in similar manner for $0<\mathtt{t}_{\mathtt{n}-2}<1,$
$$
\begin{aligned}
\int_{0}^{1}\Xi_{r_0}(\mathtt{t}_{\mathtt{n}-2}, \mathtt{t}_{\mathtt{n}-1})\ell(\mathtt{t}_{\mathtt{n}-1})\mathtt{g}_{\mathtt{n}-1}&\Bigg[ \int_{0}^{1}\Xi_{r_0}(\mathtt{t}_{\mathtt{n}-1}, \mathtt{t}_\mathtt{n})\ell(\mathtt{t}_\mathtt{n})\mathtt{g}_\mathtt{n}\big( \mathtt{u}_1(\mathtt{t}_\mathtt{n})\big)\mathtt{d}\mathtt{t}_\mathtt{n}\Bigg]\mathtt{d}\mathtt{t}_{\mathtt{n}-1}\\
&\ge\wp\int_{0}^{1}\Xi_{r_0}(\mathtt{t}_{\mathtt{n}-1}, \mathtt{t}_{\mathtt{n}-1})\ell(\mathtt{t}_{\mathtt{n}-1})\mathtt{g}_{\mathtt{n}-1}(a_1)\mathtt{d}\mathtt{t}_{\mathtt{n}-1}\\
&\ge \wp\mathfrak{Q}_1a_1\int_{0}^{1}\Xi_{r_0}(\mathtt{t}_{\mathtt{n}-1}, \mathtt{t}_{\mathtt{n}-1})\ell(\mathtt{t}_{\mathtt{n}-1})\mathtt{d}\mathtt{t}_{\mathtt{n}-1}\\
&\ge \mathfrak{Q}_1a_1\frac{\wp r_0^2}{(\mathtt{N}-2)^2}\int_{0}^{1}\Xi_{r_0}(\mathtt{t}_{\mathtt{n}-1}, \mathtt{t}_{\mathtt{n}-1})\mathtt{t}^\frac{2(\mathtt{N}-1)}{2-\mathtt{N}}_{\mathtt{n}-1}\prod_{i=1}^{m}\ell_i(\mathtt{t}_{\mathtt{n}-1})\mathtt{d}\mathtt{t}_{\mathtt{n}-1}\\
&\ge \mathfrak{Q}_1a_1\frac{\wp r_0^2}{(\mathtt{N}-2)^2}\prod_{i=1}^{m}\ell_i^\star\int_{0}^{1}\Xi_{r_0}(\mathtt{t}_{\mathtt{n}-1}, \mathtt{t}_{\mathtt{n}-1})\mathtt{t}^\frac{2(\mathtt{N}-1)}{2-\mathtt{N}}_{\mathtt{n}-1} \mathtt{d}\mathtt{t}_{\mathtt{n}-1}\\
&\ge a_1.
\end{aligned}
$$
	Continuing with bootstrapping argument, we get
	$$
	\begin{aligned}
		(\aleph  \mathtt{u}_1)(\mathtt{s})=&\,  \int_{0}^{1}\Xi_{r_0}(\mathtt{s}, \mathtt{t}_1)\ell(\mathtt{t}_1)\mathtt{g}_1\Bigg[ \int_{0}^{1}\Xi_{r_0}(\mathtt{t}_1, \mathtt{t}_2)\ell(\mathtt{t}_2)\mathtt{g}_2\Bigg[ \int_{0}^{1}\Xi_{r_0}(\mathtt{t}_2, \mathtt{t}_3)\ell(\mathtt{t}_3)\mathtt{g}_4 \cdots\\
		&\hskip0.6cm\mathtt{g}_{\mathtt{n}-1}\Bigg[ \int_{0}^{1}\Xi_{r_0}(\mathtt{t}_{\mathtt{n}-1}, \mathtt{t}_\mathtt{n})\ell(\mathtt{t}_\mathtt{n})\mathtt{g}_\mathtt{n}\big( \mathtt{u}_1(\mathtt{t}_\mathtt{n})\big)\mathtt{d}\mathtt{t}_\mathtt{n}\Bigg]\cdot\cdot\cdot \Bigg]\mathtt{d}\mathtt{t}_3\Bigg]\mathtt{d}\mathtt{t}_2\Bigg]\mathtt{d}\mathtt{t}_1\\
		\ge&\,a_1.
	\end{aligned}
	$$
	Thus, for $ \mathtt{u}_1\in \mathtt{E}\cap\partial\mathtt{F},$ we have
	\begin{equation}\label{eq42}
		\Vert \aleph  \mathtt{u}_1\Vert\ge\Vert  \mathtt{u}_1\Vert.
	\end{equation} 
	It is clear that $0\in\mathtt{F}\subset\overline{\mathtt{F}}\subset\mathtt{G}.$ It follows from \eqref{eq41}, \eqref{eq42} and  Theorem \ref{t41} that the operator $\aleph $ has a fixed point $ \mathtt{u}_1\in \mathtt{E}\cap\big(\overline{\mathtt{G}}\backslash\mathtt{F}\big)$ such that $ \mathtt{u}_1(\mathtt{s})\ge 0$ on $(0, 1).$ Next setting $ \mathtt{u}_{\mathtt{n}+1}= \mathtt{u}_1,$ we obtain infinitely many positive solutions $( \mathtt{u}_1,  \mathtt{u}_2,\cdot\cdot\cdot,\mathtt{u}_\mathtt{n})$ of \eqref{eq13}--\eqref{eq130} given iteratively by
	$$
	\begin{aligned}
		\mathtt{u}_{\dot{\iota}} (\mathtt{s})&= \int_{0}^{1}\Xi_{r_0}(\mathtt{s},s)\ell(s)\mathtt{g}_{{\dot{\iota}} }( \mathtt{u}_{{\dot{\iota}} +1}(s))ds,\,{\dot{\iota}} =1, 2,\cdot\cdot\cdot,{\mathtt{n}}-1,\mathtt{n},\\
		\mathtt{u}_{\mathtt{n}+1}(\mathtt{s})&= \mathtt{u}_1(\mathtt{s}),\,\mathtt{s}\in(0,1).
	\end{aligned}
	$$
	This completes the proof.\qed
\end{proof}

For $\sum_{i=1}^{m}\frac{1}{\mathtt{p}_i}=1$ and $\sum_{i=1}^{m}\frac{1}{\mathtt{p}_i}>1,$ we have following results.
\begin{theorem}\label{t44}
		Suppose $(\mathcal{J}_1)$--$(\mathcal{J}_3)$ hold. Further, assume that there exist two positive constants $b_2>b_1>0$ such that $\mathtt{g}_{\dot{\iota}} \,({\dot{\iota}} =1,2,\cdot\cdot\cdot,\mathtt{n})$ satisfies $(\mathcal{J}_5)$ and
	\begin{itemize}
		\item[$(\mathcal{J}_6)$] $\mathtt{g}_{\dot{\iota}} ( \mathtt{u}(\mathtt{s}))\le \mathfrak{N}_2b_2$ for all $0\le\mathtt{s}\le1,\, 0\le \mathtt{u} \le b_2,$ 	where $\displaystyle\mathfrak{N}_2=\left[\frac{ r_0^2}{(\mathtt{N}-2)^2}\Vert\widehat{\Xi}_{r_0}\Vert_\infty\prod_{i=1}^{m}\Vert\ell_i\Vert_{\mathtt{p}_i}\right]^{-1}$ and $\widehat{\Xi}_{r_0}(\mathtt{t})=\Xi_{r_0}(\mathtt{t}, \mathtt{t})\mathtt{t}^\frac{2(\mathtt{N}-1)}{2-\mathtt{N}}.$
	\end{itemize}
	Then iterative system \eqref{eq13}--\eqref{eq130} has atleast one positive radial solution $(\mathtt{u}_1,\mathtt{u}_2,\cdot\cdot\cdot,\mathtt{u}_\mathtt{n})$ such that $b_1\le\Vert\mathtt{u}_{\dot{\iota}} \Vert\le b_2,\,{\dot{\iota}} =1,2,\cdot\cdot\cdot,\mathtt{n}.$
\end{theorem}
\begin{theorem}\label{t45}
	Suppose $(\mathcal{J}_1)$--$(\mathcal{J}_3)$ hold. Further, assume that there exist two positive constants $c_2>c_1>0$ such that $\mathtt{g}_{\dot{\iota}} \,({\dot{\iota}} =1,2,\cdot\cdot\cdot,\mathtt{n})$ satisfies $(\mathcal{J}_5)$ and
	\begin{itemize}
		\item[$(\mathcal{J}_7)$] $\mathtt{g}_{\dot{\iota}} ( \mathtt{u}(\mathtt{s}))\le \mathfrak{M}_2c_2$ for all $0\le\mathtt{s}\le1,\, 0\le \mathtt{u} \le c_2,$ 	where $\displaystyle\mathfrak{M}_2=\left[\frac{ r_0^2}{(\mathtt{N}-2)^2}\Vert\widehat{\Xi}_{r_0}\Vert_\infty\prod_{i=1}^{m}\Vert\ell_i\Vert_1\right]^{-1}$ and $\widehat{\Xi}_{r_0}(\mathtt{t})=\Xi_{r_0}(\mathtt{t}, \mathtt{t})\mathtt{t}^\frac{2(\mathtt{N}-1)}{2-\mathtt{N}}.$
	\end{itemize}
	Then iterative system \eqref{eq13}--\eqref{eq130} has atleast one positive radial solution $(\mathtt{u}_1,\mathtt{u}_2,\cdot\cdot\cdot,\mathtt{u}_\mathtt{n})$ such that $c_1\le\Vert\mathtt{u}_{\dot{\iota}} \Vert\le c_2,\,{\dot{\iota}} =1,2,\cdot\cdot\cdot,\mathtt{n}.$
\end{theorem}
\begin{example}
	Consider the following nonlinear elliptic system of equations,
	\begin{equation}\label{e4}
			\triangle{  \mathtt{u}_{{\dot{\iota}} }}+\frac{(\mathtt{N}-2)^2r_0^{2\mathtt{N}-2}}{\vert x\vert^{2\mathtt{N}-2}}\mathtt{u}_{\dot{\iota}} +\ell(\vert x\vert)\mathtt{g}_{{\dot{\iota}} }(\mathtt{u}_{{\dot{\iota}} +1})=0,~1<\vert x\vert<3,
	\end{equation}
\begin{equation}\label{e5}\left.
	\begin{aligned}
		&\hskip0.9cm \mathtt{u}_{{\dot{\iota}} }=0~~\text{on}~~\vert x\vert=1~\text{and}~\vert x\vert=3,\\
		&\mathtt{u}_{{\dot{\iota}} }=0~~\text{on}~~\vert x\vert=1~\text{and}~\frac{\partial \mathtt{u}_{\dot{\iota}} }{\partial r}=0~\text{on}~\vert x\vert=3,\\
		&\frac{\partial \mathtt{u}_{\dot{\iota}} }{\partial r}=0~~\text{on}~~\vert x\vert=1~\text{and}~\mathtt{u}_{{\dot{\iota}} }=0~\text{on}~\vert x\vert=3,
	\end{aligned}\right\}
\end{equation}
\end{example}
where $ r_0=1,$ $\mathtt{N}=3,$ ${\dot{\iota}} \in\{1,2\},\, \mathtt{u}_3= \mathtt{u}_1,$  $\ell(\mathtt{s})=\frac{1}{\mathtt{s}^4}\prod_{i=1}^{2}\ell_i(\mathtt{s}),$ $\ell_i(\mathtt{s})=\ell_i\left(\frac{1}{\mathtt{s}}\right),$
in which
$$\ell_1(t)=\frac{1}{t^2+1}~~~~\text{~~and~~}~~~~\ell_2(t)=\frac{1}{\sqrt{t+2}},$$
then it is clear that
$$\ell_1, \ell_2\in L^\mathtt{p}[0,1]~~\text{~~and~~}\prod_{i=1}^{2}\ell_i^*=\sqrt{2}.$$
Let $\mathtt{g}_1(\mathtt{u})=\mathtt{g}_2(\mathtt{u})=1+\frac{1}{5}\cos(1+\mathtt{u})+\frac{1}{1+\mathtt{u}}.$ Let $\upalpha=\upbeta=\gamma=\updelta=1,$   $\varrho=2\cosh(1)+2\sinh(1)\approx5.436563658,$ 
$$
\begin{aligned}
	\Xi_{r_0}(\mathtt{s},\mathtt{t})=\frac{1}{2\cosh(1)+2\sinh(1)}\left\{\begin{array}{ll}\big(\sinh(\mathtt{s})+\cosh(\mathtt{s})\big)\big(\sinh(1-\mathtt{t})+\cosh(1-\mathtt{t})\big), \hskip0.6cm 0  \le \mathtt{s}\le \mathtt{t}\le 1,\vspace{1.2mm}\\	
	\big(\sinh(\mathtt{t})+\cosh(\mathtt{t})\big)\big(\sinh(1-\mathtt{s})+\cosh(1-\mathtt{s})\big),
		\hskip0.6cm  0  \le \mathtt{t}\le \mathtt{s}\le 1,
	\end{array}
	\right.
\end{aligned}
$$
and $\wp=\frac{1}{\sinh(1)+\cosh(1)}.$ 
Also, 
$$
\begin{aligned}
	\mathfrak{Q}_1=\left[\frac{\wp r_0^2}{(\mathtt{N}-2)^2}\prod_{i=1}^{m}\ell_i^\star\int_{0}^{1}\Xi_{r_0}(\mathtt{t}, \mathtt{t})\mathtt{t}^\frac{2(\mathtt{N}-1)}{2-\mathtt{N}} \mathtt{d}\mathtt{t}\right]^{-1}\approx0.1153270463\times10^{-4}.
\end{aligned}
$$
Let $\mathtt{p}_1=2,\mathtt{p}_2=3$ and $\mathtt{q}=6,$ then $\frac{1}{\mathtt{p}_1}+\frac{1}{\mathtt{p}_2}+\frac{1}{\mathtt{q}}=1$ and
$$
\begin{aligned}
\mathfrak{Q}_2=\left[\frac{ r_0^2}{(\mathtt{N}-2)^2}\Vert\widehat{\Xi}_{r_0}\Vert_\mathtt{q}\prod_{i=1}^{m}\Vert\ell_i\Vert_{\mathtt{p}_i}\right]^{-1}\approx0.4577977612\times10^{-7}.
\end{aligned}
$$
Choose $a_1=10^3$ and $a_2=10^8.$ Then,
$$
\begin{aligned}
\mathtt{g}_1(\mathtt{u})=\mathtt{g}_2(\mathtt{u})&\,=1+\frac{1}{5}\cos(1+\mathtt{u})+\frac{1}{1+\mathtt{u}}\le4.577977612=\mathfrak{Q}_2a_2,\,\mathtt{u}\in[0,10^8],\\
\mathtt{g}_1(\mathtt{u})=\mathtt{g}_2(\mathtt{u})&\,=1+\frac{1}{5}\cos(1+\mathtt{u})+\frac{1}{1+\mathtt{u}}\ge0.011532704=\mathfrak{Q}_1a_1,\,\mathtt{u}\in[0,10^3].
\end{aligned}
$$
Therefore, by Theorem \ref{t43}, the boundary value problem \eqref{e4}--\eqref{e5} has at least one positive solution $(\mathtt{u}_1,\mathtt{u}_2)$ such that $10^3\le\Vert  \mathtt{u}_{\dot{\iota}} \Vert\le10^8$ for ${\dot{\iota}} =1,2.$
\section{Existence of at Least  Two Positive Radial Solutions}
In this section, we establish the existence of at least two positive radial solutions for the system \eqref{eq13}--\eqref{eq130} by an application of following Avery-Henderson fixed point theorem. 

Let $\uppsi$ be a nonnegative continuous functional on a cone $\mathtt{E}$ of the real Banach space $\mathcal{B}.$ Then for a positive real numbers $a'$ and $c',$ we define the sets
$$\mathtt{E}(\uppsi,c')=\{\mathtt{u}\in\mathtt{E}:\uppsi(\mathtt{u})<c'\},$$
and
$$\mathtt{E}_{a'}=\{\mathtt{u}\in\mathtt{E}:\Vert\mathtt{u}\Vert<a'\}.$$
\begin{theorem}(Avery-Henderson\cite{avery})\label{t411}
	Let $\mathtt{E}$ be a cone in a real Banach space $\mathtt{B}.$ Suppose $\ss_1$ and $\ss_2$ are increasing, nonnegative continuous functionals on $\mathtt{E}$ and $\ss_3$ is nonnegative continuous functional on $\mathtt{E}$ with $\ss_3(0)=0$ such that, for some positive numbers $c'$ and $k,$ $\ss_2(\mathtt{u})\le\ss_3(\mathtt{u})\le\ss_1(\mathtt{u})$ and $\Vert\mathtt{u}\Vert\le k\ss_2(\mathtt{u}),$ for all $\mathtt{u}\in \overline{\mathtt{E}(\ss_2,c')}.$ Suppose that there exist positive numbers $a'$ and $b'$ with $a'<b'<c'$ such that $\ss_3(\uplambda\mathtt{u})\le\uplambda\ss_3(\mathtt{u}),$ for all $0\le\uplambda\le 1$ and $\mathtt{u}\in\partial\mathtt{E}(\ss_3,b').$ Further, let $\aleph :\overline{\mathtt{E}(\ss_2,c')}\to\mathtt{E}$ be a completely continuous operator such that
	\begin{itemize}
		\item [$(a)$] $\ss_2(\aleph \mathtt{u})>c',$ for all $\mathtt{u}\in \partial\mathtt{E}(\ss_2,c'),$
		\item [$(b)$] $\ss_3(\aleph \mathtt{u})<b',$ for all $\mathtt{u}\in \partial\mathtt{E}(\ss_3,b'),$
		\item [$(c)$] $\mathtt{E}(\ss_1,a')\ne\emptyset$ and $\ss_1(\aleph \mathtt{u})>a',$ for all $\partial\mathtt{E}(\ss_1,a').$
	\end{itemize}
	Then, $\aleph $ has at least two fixed points ${}^1\mathtt{u},{}^2\mathtt{u}\in\mathtt{P}(\ss_2,c')$ such that
	$a'<\ss_1({}^1\mathtt{u})$ with $\ss_3({}^1\mathtt{u})<b'$ and $b'<\ss_3({}^2\mathtt{u})$ with $\ss_2({}^2\mathtt{u})<c'.$
\end{theorem}
Define the nonnegative, increasing, continuous functional $\ss_2, \ss_3,$ and $\ss_1$ by
$$
\begin{aligned}
	\ss_2(\mathtt{u})=\min_{\mathtt{s}\in[0,1]}\mathtt{u}(\mathtt{s}),\,
	\ss_3(\mathtt{u})=\max_{\mathtt{s}\in[0,1]}\mathtt{u}(\mathtt{s}),\,
	\ss_1(\mathtt{u})=\max_{\mathtt{s}\in[0,1]}\mathtt{u}(\mathtt{s}).
\end{aligned}
$$
It is obvious that for each $\mathtt{u}\in\mathtt{E},$
$$\ss_2(\mathtt{u})\le\ss_3(\mathtt{u})=\ss_1(\mathtt{u}).$$
In addition, by Lemma \ref{l21}, for each $\mathtt{u}\in\mathtt{P},$
$$\ss_2(\mathtt{u})\ge\wp\Vert\mathtt{u}\Vert.$$
Thus 
$$\Vert\mathtt{u}\Vert\le\frac{1}{\wp}\ss_2(\mathtt{u})~~\textnormal{for~all}~~\mathtt{u}\in\mathtt{E}.$$
Finally, we also note that
$$\ss_3(\uplambda\mathtt{u})=\uplambda\ss_3(\mathtt{u}),~~0\le\uplambda\le 1~~\textnormal{and}~~\mathtt{u}\in\mathtt{E}.$$
\begin{theorem}\label{t431}
	Assume that $(\mathcal{J}_1)$--$(\mathcal{J}_3)$ hold and Suppose there exist real numbers $a', b'$ and $c'$ with $0<a'<b'<c'$ such that $\mathtt{g}_{\dot{\iota}} ({\dot{\iota}} =1,2,\cdot\cdot\cdot,\mathtt{n})$ satisfies
	\begin{itemize}
		\item [$(\mathcal{J}_{8})$] $\mathtt{g}_{\dot{\iota}} (\mathtt{u})>\frac{ c'}{\Bbbk_1}$, for all $c'\le\mathtt{u}\le\frac{c'}{\wp},$ where $\Bbbk_1=\frac{\wp r_0^2}{(\mathtt{N}-2)^2}\prod_{i=1}^{m}\ell_i^\star\int_{0}^{1}\Xi_{r_0}(\mathtt{t}, \mathtt{t})\mathtt{t}^\frac{2(\mathtt{N}-1)}{2-\mathtt{N}} \mathtt{d}\mathtt{t},$
		\item [$(\mathcal{J}_{9})$] $\mathtt{g}_{\dot{\iota}} (\mathtt{u})<\frac{b'}{\Bbbk_2}$, for all $0\le\mathtt{u}\le\frac{b'}{\wp},$ where $\Bbbk_2=\frac{ r_0^2}{(\mathtt{N}-2)^2}\Vert\widehat{\Xi}_{r_0}\Vert_\mathtt{q}\prod_{i=1}^{m}\Vert\ell_i\Vert_{\mathtt{p}_i},$
		\item [$(\mathcal{J}_{10})$] $\mathtt{g}_{\dot{\iota}} (\mathtt{u})>\frac{a'}{\Bbbk_1}$, for all $a'\le\mathtt{u}\le\frac{a'}{\wp}.$
	\end{itemize}
	Then the boundary value problem \eqref{eq13}--\eqref{eq130} has at least two positive radial solutions $\{({}^1\mathtt{u}_1,{}^1\mathtt{u}_2,\cdot\cdot\cdot,{}^1\mathtt{u}_\mathtt{n})\}$ and  $\{({}^2\mathtt{u}_1,{}^2\mathtt{u}_2,\cdot\cdot\cdot,{}^2\mathtt{u}_\mathtt{n})\}$ satisfying
	$$a'<\ss_1\big({}^1\mathtt{u}_{{\dot{\iota}} }\big)~~\textit{with}~~\ss_3\big({}^1\mathtt{u}_{{\dot{\iota}} }\big)<b',~{\dot{\iota}} =1,2,\cdot\cdot\cdot,\mathtt{n},$$
	and
	$$b'<\ss_3\big({}^2\mathtt{u}_{{\dot{\iota}} }\big)~~\textit{with}~~\ss_2\big({}^2\mathtt{u}_{{\dot{\iota}} }\big)<c',~{\dot{\iota}} =1,2,\cdot\cdot\cdot,\mathtt{n}.$$
\end{theorem}
\begin{proof}
	We begin by defining the completely continuous operator $\aleph $ by \eqref{eqmain}. So it is easy to check that $\aleph :\overline{\mathtt{E}(\ss_2,c')}\to\mathtt{E}.$ Firstly, we shall verify that condition $(a)$ of Theorem \ref{t411} is satisfied. So, let us choose $\mathtt{u}_1\in\partial\mathtt{E}(\ss_2,c').$ Then $\ss_2(\mathtt{u}_1)=\min_{\mathtt{s}\in[0,1]}\mathtt{u}_1(\mathtt{s})=c'$ this implies that $c'\le\mathtt{u}_1(\mathtt{s})$ for
	$\mathtt{s}\in[0,1].$ Since $\Vert\mathtt{u}_1\Vert\le\frac{1}{\wp}\ss_2(\mathtt{u}_1)=\frac{1}{\wp}c'.$ So we have
	$$c'\le\mathtt{u}_1(\mathtt{s})\le\frac{c'}{\wp},~\mathtt{s}\in[0,1].$$
	Let $\mathtt{t}_{\mathtt{n}-1}\in[0,1].$ Then by $(\mathcal{J}_{8}),$ we have
$$
\begin{aligned}
	\int_{0}^{1}\Xi_{r_0}(\mathtt{t}_{\mathtt{n}-1}, \mathtt{t}_\mathtt{n})\ell(\mathtt{t}_\mathtt{n})\mathtt{g}_\mathtt{n}\big( \mathtt{u}_1(\mathtt{t}_\mathtt{n})\big)\mathtt{d}\mathtt{t}_\mathtt{n}
	&\ge \wp\int_{0}^{1}\Xi_{r_0}(\mathtt{t}_{\mathtt{n}}, \mathtt{t}_\mathtt{n})\ell(\mathtt{t}_\mathtt{n})\mathtt{g}_\mathtt{n}\big( \mathtt{u}_1(\mathtt{t}_\mathtt{n})\big)\mathtt{d}\mathtt{t}_\mathtt{n}\\
	&\ge \frac{\wp c'}{\Bbbk_1}\int_{0}^{1}\Xi_{r_0}(\mathtt{t}_{\mathtt{n}}, \mathtt{t}_\mathtt{n})\ell(\mathtt{t}_\mathtt{n})\mathtt{d}\mathtt{t}_\mathtt{n}\\
	&\ge \frac{\wp c'r_0^2}{(\mathtt{N}-2)^2\Bbbk_1}\int_{0}^{1}\Xi_{r_0}(\mathtt{t}_{\mathtt{n}}, \mathtt{t}_\mathtt{n})\mathtt{t}^\frac{2(\mathtt{N}-1)}{2-\mathtt{N}}_\mathtt{n}\prod_{i=1}^{m}\ell_i(\mathtt{t}_\mathtt{n})\mathtt{d}\mathtt{t}_\mathtt{n}\\
	&\ge \frac{\wp c'r_0^2}{(\mathtt{N}-2)^2\Bbbk_1}\prod_{i=1}^{m}\ell_i^\star\int_{0}^{1}\Xi_{r_0}(\mathtt{t}_{\mathtt{n}}, \mathtt{t}_\mathtt{n})\mathtt{t}^\frac{2(\mathtt{N}-1)}{2-\mathtt{N}}_\mathtt{n} \mathtt{d}\mathtt{t}_\mathtt{n}\\
	&\ge c'.
\end{aligned}
$$
It follows in similar manner for $0<\mathtt{t}_{\mathtt{n}-2}<1,$
$$
\begin{aligned}
	\int_{0}^{1}\Xi_{r_0}(\mathtt{t}_{\mathtt{n}-2}, \mathtt{t}_{\mathtt{n}-1})\ell(\mathtt{t}_{\mathtt{n}-1})\mathtt{g}_{\mathtt{n}-1}&\Bigg[ \int_{0}^{1}\Xi_{r_0}(\mathtt{t}_{\mathtt{n}-1}, \mathtt{t}_\mathtt{n})\ell(\mathtt{t}_\mathtt{n})\mathtt{g}_\mathtt{n}\big( \mathtt{u}_1(\mathtt{t}_\mathtt{n})\big)\mathtt{d}\mathtt{t}_\mathtt{n}\Bigg]\mathtt{d}\mathtt{t}_{\mathtt{n}-1}\\
	&\ge\wp\int_{0}^{1}\Xi_{r_0}(\mathtt{t}_{\mathtt{n}-1}, \mathtt{t}_{\mathtt{n}-1})\ell(\mathtt{t}_{\mathtt{n}-1})\mathtt{g}_{\mathtt{n}-1}(c')\mathtt{d}\mathtt{t}_{\mathtt{n}-1}\\
	&\ge \frac{\wp c'}{\Bbbk_1}\int_{0}^{1}\Xi_{r_0}(\mathtt{t}_{\mathtt{n}-1}, \mathtt{t}_{\mathtt{n}-1})\ell(\mathtt{t}_{\mathtt{n}-1})\mathtt{d}\mathtt{t}_{\mathtt{n}-1}\\
	&\ge \frac{\wp c'r_0^2}{(\mathtt{N}-2)^2\Bbbk_1}\int_{0}^{1}\Xi_{r_0}(\mathtt{t}_{\mathtt{n}-1}, \mathtt{t}_{\mathtt{n}-1})\mathtt{t}^\frac{2(\mathtt{N}-1)}{2-\mathtt{N}}_{\mathtt{n}-1}\prod_{i=1}^{m}\ell_i(\mathtt{t}_{\mathtt{n}-1})\mathtt{d}\mathtt{t}_{\mathtt{n}-1}\\
	&\ge \frac{\wp c'r_0^2}{(\mathtt{N}-2)^2\Bbbk_1}\prod_{i=1}^{m}\ell_i^\star\int_{0}^{1}\Xi_{r_0}(\mathtt{t}_{\mathtt{n}-1}, \mathtt{t}_{\mathtt{n}-1})\mathtt{t}^\frac{2(\mathtt{N}-1)}{2-\mathtt{N}}_{\mathtt{n}-1} \mathtt{d}\mathtt{t}_{\mathtt{n}-1}\\
	&\ge c'.
\end{aligned}
$$
Continuing with bootstrapping argument, we get
$$
\begin{aligned}
\ss_2\left(\aleph  \mathtt{u}_1\right)=&\min_{\mathtt{s}\in[0,1]}\int_{0}^{1}\Xi_{r_0}(\mathtt{s}, \mathtt{t}_1)\ell(\mathtt{t}_1)\mathtt{g}_1\Bigg[ \int_{0}^{1}\Xi_{r_0}(\mathtt{t}_1, \mathtt{t}_2)\ell(\mathtt{t}_2)\mathtt{g}_2\Bigg[ \int_{0}^{1}\Xi_{r_0}(\mathtt{t}_2, \mathtt{t}_3)\ell(\mathtt{t}_3)\mathtt{g}_4 \cdots\\
	&\hskip1.6cm\mathtt{g}_{\mathtt{n}-1}\Bigg[ \int_{0}^{1}\Xi_{r_0}(\mathtt{t}_{\mathtt{n}-1}, \mathtt{t}_\mathtt{n})\ell(\mathtt{t}_\mathtt{n})\mathtt{g}_\mathtt{n}\big( \mathtt{u}_1(\mathtt{t}_\mathtt{n})\big)\mathtt{d}\mathtt{t}_\mathtt{n}\Bigg]\cdot\cdot\cdot \Bigg]\mathtt{d}\mathtt{t}_3\Bigg]\mathtt{d}\mathtt{t}_2\Bigg]\mathtt{d}\mathtt{t}_1\\
	\ge&\,c'.
\end{aligned}
$$
	This proves $(i)$ of Theorem \ref{t411}. We next address $(ii)$ of Theorem \ref{t411}. So, we choose $\mathtt{u}_1\in\partial\mathtt{E}(\ss_3,b').$ Then $\ss_3(\mathtt{u}_1)=\max_{\mathtt{s}\in[0,1]}\mathtt{u}_1(\mathtt{s})=b'$ this implies that $0\le\mathtt{u}_1(\mathtt{s})\le b'$ for
	$\mathtt{s}\in[0,1].$ Since $\Vert\mathtt{u}_1\Vert\le\frac{1}{\wp}\ss_2(\mathtt{u}_1)\le\frac{1}{\wp}\ss_3(\mathtt{u}_1)=\frac{b'}{\wp}.$ So we have
	$$0\le\mathtt{u}_1(\mathtt{s})\le\wp^2b',~\mathtt{s}\in[0,1].$$
	Let $0<\mathtt{t}_{\mathtt{n}-1}<1.$ Then by $(\mathcal{J}_{9}),$ we have
	$$
\begin{aligned}
	\int_{0}^{1}\Xi_{r_0}(\mathtt{t}_{\mathtt{n}-1}, \mathtt{t}_\mathtt{n})\ell(\mathtt{t}_\mathtt{n})\mathtt{g}_\mathtt{n}\big( \mathtt{u}_1(\mathtt{t}_\mathtt{n})\big)\mathtt{d}\mathtt{t}_\mathtt{n}
	&\le \wp\int_{0}^{1}\Xi_{r_0}(\mathtt{t}_{\mathtt{n}}, \mathtt{t}_\mathtt{n})\ell(\mathtt{t}_\mathtt{n})\mathtt{g}_\mathtt{n}\big( \mathtt{u}_1(\mathtt{t}_\mathtt{n})\big)\mathtt{d}\mathtt{t}_\mathtt{n}\\
	&\le \frac{\wp b'}{\Bbbk_2}\int_{0}^{1}\Xi_{r_0}(\mathtt{t}_{\mathtt{n}}, \mathtt{t}_\mathtt{n})\ell(\mathtt{t}_\mathtt{n})\mathtt{d}\mathtt{t}_\mathtt{n}\\
	&\le \frac{\wp b' r_0^2}{(\mathtt{N}-2)^2\Bbbk_2}\int_{0}^{1}\Xi_{r_0}(\mathtt{t}_{\mathtt{n}}, \mathtt{t}_\mathtt{n})\mathtt{t}^\frac{2(\mathtt{N}-1)}{2-\mathtt{N}}_\mathtt{n}\prod_{i=1}^{m}\ell_i(\mathtt{t}_\mathtt{n})\mathtt{d}\mathtt{t}_\mathtt{n}.
\end{aligned}
$$
There exists a $\mathtt{q}>1$ such that $\displaystyle\sum_{i=1}^{m}\frac{1}{\mathtt{p}_i}+\frac{1}{\mathtt{q}}=1.$ By the first part of Theorem \ref{holder}, we have
$$
\begin{aligned}
	\int_{0}^{1}\Xi_{r_0}(\mathtt{t}_{\mathtt{n}-1}, \mathtt{t}_\mathtt{n})\ell(\mathtt{t}_\mathtt{n})\mathtt{g}_\mathtt{n}\big( \mathtt{u}_1(\mathtt{t}_\mathtt{n})\big)\mathtt{d}\mathtt{t}_\mathtt{n}&\le \frac{\wp b' r_0^2}{(\mathtt{N}-2)^2\Bbbk_2}\Vert\widehat{\Xi}_{r_0}\Vert_\mathtt{q}\prod_{i=1}^{m}\Vert\ell_i\Vert_{\mathtt{p}_i}\\
	&\le b'.
\end{aligned}
$$
Continuing with this bootstrapping argument, we get
$$
\begin{aligned}
	\ss_3\left(\aleph  \mathtt{u}_1\right)=&\max_{\mathtt{s}\in[0,1]} \int_{0}^{1}\Xi_{r_0}(\mathtt{s}, \mathtt{t}_1)\ell(\mathtt{t}_1)\mathtt{g}_1\Bigg[ \int_{0}^{1}\Xi_{r_0}(\mathtt{t}_1, \mathtt{t}_2)\ell(\mathtt{t}_2)\mathtt{g}_2\Bigg[ \int_{0}^{1}\Xi_{r_0}(\mathtt{t}_2, \mathtt{t}_3)\ell(\mathtt{t}_3)\mathtt{g}_4 \cdots\\
	&\hskip1cm\mathtt{g}_{\mathtt{n}-1}\Bigg[ \int_{0}^{1}\Xi_{r_0}(\mathtt{t}_{\mathtt{n}-1}, \mathtt{t}_\mathtt{n})\ell(\mathtt{t}_\mathtt{n})\mathtt{g}_\mathtt{n}\big( \mathtt{u}_1(\mathtt{t}_\mathtt{n})\big)\mathtt{d}\mathtt{t}_\mathtt{n}\Bigg]\cdot\cdot\cdot \Bigg]\mathtt{d}\mathtt{t}_3\Bigg]\mathtt{d}\mathtt{t}_2\Bigg]\mathtt{d}\mathtt{t}_1\\
	\le&\,b'.
\end{aligned}
$$
	Hence condition $(b)$ is satisfied. Finally, we verify that $(c)$ of Theorem \ref{t411} is also satisfied. We note that $\mathtt{u}_1(\mathtt{s})=a'/4,~\mathtt{s}\in[0,1]$ is a
	member of $\mathtt{E}(\ss_1,a')$ and $a'/4<a'.$ So $\mathtt{E}(\ss_1,a')\ne\emptyset.$ Now let $\mathtt{u}_1\in\mathtt{E}(\ss_1,a').$ Then $a'=\ss_1(\mathtt{u}_1)=\max_{\mathtt{s}\in[0,1]}\mathtt{u}_1(\mathtt{s})=\Vert \mathtt{u}_1\Vert=\frac{1}{\wp}\ss_2(\mathtt{u}_1)\le\frac{1}{\wp}\ss_3(\mathtt{u}_1)=\frac{1}{\wp}\ss_1(\mathtt{u}_1)=\frac{a'}{\wp},$ i.e., $a'\le\mathtt{u}_1(\mathtt{s})\le\frac{a'}{\wp}$ for $\mathtt{s}\in[0,1].$ 
	Let $0<\mathtt{t}_{\mathtt{n}-1}<1.$ Then by $(\mathcal{J}_{10}),$ we have
	$$
	\begin{aligned}
		\int_{0}^{1}\Xi_{r_0}(\mathtt{t}_{\mathtt{n}-1}, \mathtt{t}_\mathtt{n})\ell(\mathtt{t}_\mathtt{n})\mathtt{g}_\mathtt{n}\big( \mathtt{u}_1(\mathtt{t}_\mathtt{n})\big)\mathtt{d}\mathtt{t}_\mathtt{n}
		&\ge \wp\int_{0}^{1}\Xi_{r_0}(\mathtt{t}_{\mathtt{n}}, \mathtt{t}_\mathtt{n})\ell(\mathtt{t}_\mathtt{n})\mathtt{g}_\mathtt{n}\big( \mathtt{u}_1(\mathtt{t}_\mathtt{n})\big)\mathtt{d}\mathtt{t}_\mathtt{n}\\
		&\ge \frac{\wp a'}{\Bbbk_1}\int_{0}^{1}\Xi_{r_0}(\mathtt{t}_{\mathtt{n}}, \mathtt{t}_\mathtt{n})\ell(\mathtt{t}_\mathtt{n})\mathtt{d}\mathtt{t}_\mathtt{n}\\
		&\ge \frac{\wp a'r_0^2}{(\mathtt{N}-2)^2\Bbbk_1}\int_{0}^{1}\Xi_{r_0}(\mathtt{t}_{\mathtt{n}}, \mathtt{t}_\mathtt{n})\mathtt{t}^\frac{2(\mathtt{N}-1)}{2-\mathtt{N}}_\mathtt{n}\prod_{i=1}^{m}\ell_i(\mathtt{t}_\mathtt{n})\mathtt{d}\mathtt{t}_\mathtt{n}\\
		&\ge \frac{\wp a'r_0^2}{(\mathtt{N}-2)^2\Bbbk_1}\prod_{i=1}^{m}\ell_i^\star\int_{0}^{1}\Xi_{r_0}(\mathtt{t}_{\mathtt{n}}, \mathtt{t}_\mathtt{n})\mathtt{t}^\frac{2(\mathtt{N}-1)}{2-\mathtt{N}}_\mathtt{n} \mathtt{d}\mathtt{t}_\mathtt{n}\\
		&\ge a'.
	\end{aligned}
	$$
	Continuing with this bootstrapping argument, we get
$$
\begin{aligned}
	\ss_1\left(\aleph  \mathtt{u}_1\right)=&\max_{\mathtt{s}\in[0,1]} \int_{0}^{1}\Xi_{r_0}(\mathtt{s}, \mathtt{t}_1)\ell(\mathtt{t}_1)\mathtt{g}_1\Bigg[ \int_{0}^{1}\Xi_{r_0}(\mathtt{t}_1, \mathtt{t}_2)\ell(\mathtt{t}_2)\mathtt{g}_2\Bigg[ \int_{0}^{1}\Xi_{r_0}(\mathtt{t}_2, \mathtt{t}_3)\ell(\mathtt{t}_3)\mathtt{g}_4 \cdots\\
	&\hskip1cm\mathtt{g}_{\mathtt{n}-1}\Bigg[ \int_{0}^{1}\Xi_{r_0}(\mathtt{t}_{\mathtt{n}-1}, \mathtt{t}_\mathtt{n})\ell(\mathtt{t}_\mathtt{n})\mathtt{g}_\mathtt{n}\big( \mathtt{u}_1(\mathtt{t}_\mathtt{n})\big)\mathtt{d}\mathtt{t}_\mathtt{n}\Bigg]\cdot\cdot\cdot \Bigg]\mathtt{d}\mathtt{t}_3\Bigg]\mathtt{d}\mathtt{t}_2\Bigg]\mathtt{d}\mathtt{t}_1\\
	\ge&\min_{\mathtt{s}\in[0,1]} \int_{0}^{1}\Xi_{r_0}(\mathtt{s}, \mathtt{t}_1)\ell(\mathtt{t}_1)\mathtt{g}_1\Bigg[ \int_{0}^{1}\Xi_{r_0}(\mathtt{t}_1, \mathtt{t}_2)\ell(\mathtt{t}_2)\mathtt{g}_2\Bigg[ \int_{0}^{1}\Xi_{r_0}(\mathtt{t}_2, \mathtt{t}_3)\ell(\mathtt{t}_3)\mathtt{g}_4 \cdots\\
	&\hskip1cm\mathtt{g}_{\mathtt{n}-1}\Bigg[ \int_{0}^{1}\Xi_{r_0}(\mathtt{t}_{\mathtt{n}-1}, \mathtt{t}_\mathtt{n})\ell(\mathtt{t}_\mathtt{n})\mathtt{g}_\mathtt{n}\big( \mathtt{u}_1(\mathtt{t}_\mathtt{n})\big)\mathtt{d}\mathtt{t}_\mathtt{n}\Bigg]\cdot\cdot\cdot \Bigg]\mathtt{d}\mathtt{t}_3\Bigg]\mathtt{d}\mathtt{t}_2\Bigg]\mathtt{d}\mathtt{t}_1\\
	\ge&\,a'.
\end{aligned}
$$
	Thus condition $(c)$ of Theorem \ref{t411} is satisfied. Since all hypotheses of Theorem \ref{t411} are satisfied, the assertion follows.\qed
\end{proof}	
For $\sum_{i=1}^{m}\frac{1}{\mathtt{p}_i}=1$ and $\sum_{i=1}^{m}\frac{1}{\mathtt{p}_i}>1,$ we have following results.
\begin{theorem}
	Assume that $(\mathcal{J}_1)$--$(\mathcal{J}_3)$ hold and Suppose there exist real numbers $a', b'$ and $c'$ with $0<a'<b'<c'$ such that $\mathtt{g}_{\dot{\iota}} ({\dot{\iota}} =1,2,\cdot\cdot\cdot,\mathtt{n})$ satisfies $(\mathcal{J}_{8}),$ $(\mathcal{J}_{10})$ and
	\begin{itemize}
		\item [$(\mathcal{J}_{9}')$] $\mathtt{g}_{\dot{\iota}} (\mathtt{u})<\frac{b'}{\Bbbk_3}$, for all $0\le\mathtt{u}\le\frac{b'}{\wp},$ where $\Bbbk_3=\frac{  r_0^2}{(\mathtt{N}-2)^2}\Vert\widehat{\Xi}_{r_0}\Vert_\infty\prod_{i=1}^{m}\Vert\ell_i\Vert_{\mathtt{p}_i}.$
	\end{itemize}
	Then the boundary value problem \eqref{eq13}--\eqref{eq130} has at least two positive radial solutions $\{({}^1\mathtt{u}_1,{}^1\mathtt{u}_2,\cdot\cdot\cdot,{}^1\mathtt{u}_\mathtt{n})\}$ and  $\{({}^2\mathtt{u}_1,{}^2\mathtt{u}_2,\cdot\cdot\cdot,{}^2\mathtt{u}_\mathtt{n})\}$ satisfying
	$$a'<\ss_1\big({}^1\mathtt{u}_{{\dot{\iota}} }\big)~~\textit{with}~~\ss_3\big({}^1\mathtt{u}_{{\dot{\iota}} }\big)<b',~{\dot{\iota}} =1,2,\cdot\cdot\cdot,\mathtt{n},$$
	and
	$$b'<\ss_3\big({}^2\mathtt{u}_{{\dot{\iota}} }\big)~~\textit{with}~~\ss_2\big({}^2\mathtt{u}_{{\dot{\iota}} }\big)<c',~{\dot{\iota}} =1,2,\cdot\cdot\cdot,\mathtt{n}.$$
\end{theorem}
\begin{theorem}
	Assume that $(\mathcal{J}_1)$--$(\mathcal{J}_3)$ hold and Suppose there exist real numbers $a', b'$ and $c'$ with $0<a'<b'<c'$ such that $\mathtt{g}_{\dot{\iota}} ({\dot{\iota}} =1,2,\cdot\cdot\cdot,\mathtt{n})$ satisfies $(\mathcal{J}_{8}),$ $(\mathcal{J}_{10})$ and
	\begin{itemize}
		\item [$(\mathcal{J}_{9}'')$] $\mathtt{g}_{\dot{\iota}} (\mathtt{u})<\frac{b'}{\Bbbk_4}$, for all $0\le\mathtt{u}\le\frac{b'}{\wp},$ where $\Bbbk_4=\frac{  r_0^2}{(\mathtt{N}-2)^2}\Vert\widehat{\Xi}_{r_0}\Vert_\infty\prod_{i=1}^{m}\Vert\ell_i\Vert_{1}.$
	\end{itemize}
	Then the boundary value problem \eqref{eq13}--\eqref{eq130} has at least two positive radial solutions $\{({}^1\mathtt{u}_1,{}^1\mathtt{u}_2,\cdot\cdot\cdot,{}^1\mathtt{u}_\mathtt{n})\}$ and  $\{({}^2\mathtt{u}_1,{}^2\mathtt{u}_2,\cdot\cdot\cdot,{}^2\mathtt{u}_\mathtt{n})\}$ satisfying
	$$a'<\ss_1\big({}^1\mathtt{u}_{{\dot{\iota}} }\big)~~\textit{with}~~\ss_3\big({}^1\mathtt{u}_{{\dot{\iota}} }\big)<b',~{\dot{\iota}} =1,2,\cdot\cdot\cdot,\mathtt{n},$$
	and
	$$b'<\ss_3\big({}^2\mathtt{u}_{{\dot{\iota}} }\big)~~\textit{with}~~\ss_2\big({}^2\mathtt{u}_{{\dot{\iota}} }\big)<c',~{\dot{\iota}} =1,2,\cdot\cdot\cdot,\mathtt{n}.$$
\end{theorem}
\begin{example}
	Consider the following nonlinear elliptic system of equations,
	\begin{equation}\label{e6}
		\triangle{  \mathtt{u}_{{\dot{\iota}} }}+\frac{(\mathtt{N}-2)^2r_0^{2\mathtt{N}-2}}{\vert x\vert^{2\mathtt{N}-2}}\mathtt{u}_{\dot{\iota}} +\ell(\vert x\vert)\mathtt{g}_{{\dot{\iota}} }(\mathtt{u}_{{\dot{\iota}} +1})=0,~1<\vert x\vert<3,
	\end{equation}
	\begin{equation}\label{e7}\left.
		\begin{aligned}
			&\hskip0.9cm \mathtt{u}_{{\dot{\iota}} }=0~~\text{on}~~\vert x\vert=1~\text{and}~\vert x\vert=3,\\
			&\mathtt{u}_{{\dot{\iota}} }=0~~\text{on}~~\vert x\vert=1~\text{and}~\frac{\partial \mathtt{u}_{\dot{\iota}} }{\partial r}=0~\text{on}~\vert x\vert=3,\\
			&\frac{\partial \mathtt{u}_{\dot{\iota}} }{\partial r}=0~~\text{on}~~\vert x\vert=1~\text{and}~\mathtt{u}_{{\dot{\iota}} }=0~\text{on}~\vert x\vert=3,
		\end{aligned}\right\}
	\end{equation}
\end{example}
where $ r_0=1,$ $\mathtt{N}=3,$ ${\dot{\iota}} \in\{1,2\},\, \mathtt{u}_3= \mathtt{u}_1,$  $\ell(\mathtt{s})=\frac{1}{\mathtt{s}^4}\prod_{i=1}^{2}\ell_i(\mathtt{s}),$ $\ell_i(\mathtt{s})=\ell_i\left(\frac{1}{\mathtt{s}}\right),$
in which
$$\ell_1(t)=\frac{1}{t+1}~~~~\text{~~and~~}~~~~\ell_2(t)=\frac{1}{t^2+1},$$
then it is clear that
$$\ell_1, \ell_2\in L^\mathtt{p}[0,1]~~\text{~~and~~}\prod_{i=1}^{2}\ell_i^*=1.$$
Let $\upalpha=\upbeta=\gamma=\updelta=1,$ then  $\varrho=2\cosh(1)+2\sinh(1)\approx5.436563658,$ 
$$
\begin{aligned}
	\Xi_{r_0}(\mathtt{s},\mathtt{t})=\frac{1}{2\cosh(1)+2\sinh(1)}\left\{\begin{array}{ll}\big(\sinh(\mathtt{s})+\cosh(\mathtt{s})\big)\big(\sinh(1-\mathtt{t})+\cosh(1-\mathtt{t})\big), \hskip0.6cm 0  \le \mathtt{s}\le \mathtt{t}\le 1,\vspace{1.2mm}\\	
		\big(\sinh(\mathtt{t})+\cosh(\mathtt{t})\big)\big(\sinh(1-\mathtt{s})+\cosh(1-\mathtt{s})\big),
		\hskip0.6cm  0  \le \mathtt{t}\le \mathtt{s}\le 1,
	\end{array}
	\right.
\end{aligned}
$$
and $\wp=\frac{1}{\sinh(1)+\cosh(1)}.$ 
Also, 
$$
\begin{aligned}
\Bbbk_1=\frac{\wp r_0^2}{(\mathtt{N}-2)^2}\prod_{i=1}^{m}\ell_i^\star\int_{0}^{1}\Xi_{r_0}(\mathtt{t}, \mathtt{t})\mathtt{t}^\frac{2(\mathtt{N}-1)}{2-\mathtt{N}} \mathtt{d}\mathtt{t}\approx0.1630970729\times10^{-4}.
\end{aligned}
$$
Let $\mathtt{p}_1=3,\mathtt{p}_2=6$ and $\mathtt{q}=2,$ then $\frac{1}{\mathtt{p}_1}+\frac{1}{\mathtt{p}_2}+\frac{1}{\mathtt{q}}=1$ and
$$
\begin{aligned}
\Bbbk_2=\frac{  r_0^2}{(\mathtt{N}-2)^2}\Vert\widehat{\Xi}_{r_0}\Vert_\mathtt{q}\prod_{i=1}^{m}\Vert\ell_i\Vert_{\mathtt{p}_i}\approx4.388193758\times10^{-8}.
\end{aligned}
$$
Let $$
\begin{aligned}
	\mathtt{g}_1(\mathtt{u})=\mathtt{g}_2(\mathtt{u})=\left\{\begin{array}{ll}10^{16}, \hskip1.9cm  \mathtt{u}\ge1,\vspace{1.2mm}\\	
		10^{16}\mathtt{u}^2-\mathtt{u}+1, \hskip0.4cm  \mathtt{u}<1.
	\end{array}
	\right.
\end{aligned}
$$
Choose $a'=10^4,\,b'=10^9$ and $c'=10^{10}.$ Then, 
$$
\begin{aligned}
\mathtt{g}_1(\mathtt{u})=\mathtt{g}_2(\mathtt{u})&\ge6.131317885\times10^{14}=\frac{c'}{\Bbbk_1},~~\mathtt{u}\in[10^{10},2.72\times10^{10}],\\
\mathtt{g}_1(\mathtt{u})=\mathtt{g}_2(\mathtt{u})&\le2.278841945\times 10^{16}=\frac{b'}{\Bbbk_2},~~\mathtt{u}\in[0,2.72\times10^9],\\
\mathtt{g}_1(\mathtt{u})=\mathtt{g}_2(\mathtt{u})&\ge6.131317885\times10^8=\frac{a'}{\Bbbk_1},~~\mathtt{u}\in[10^4,2.72\times10^4].
\end{aligned}
$$
Therefore, by Theorem \ref{t43}, the boundary value problem \eqref{e6}--\eqref{e7} has at least two positive radial solutions $({}^{\dot{\iota}} \mathtt{u}_1,{}^{\dot{\iota}} \mathtt{u}_2),\,{\dot{\iota}} =1,2$ such that 
$$10^4<\max_{\mathtt{s}\in[0,1]}{}^{\dot{\iota}} \mathtt{u}_1(\mathtt{s})~~\text{with}~~\max_{\mathtt{s}\in[0,1]}{}^{\dot{\iota}} \mathtt{u}_1(\mathtt{s})<10^9,~~\text{for}~~{\dot{\iota}} =1,2,$$
$$10^9<\max_{\mathtt{s}\in[0,1]}{}^{\dot{\iota}} \mathtt{u}_2(\mathtt{s})~~\text{with}~~\min_{\mathtt{s}\in[0,1]}{}^{\dot{\iota}} \mathtt{u}_2(\mathtt{s})<10^{10},~~\text{for}~~{\dot{\iota}} =1,2.$$
\section{Existence of at Least Three Positive Radial Solutions}
In this section, we establish the existence of at least three positive radial solutions for the system \eqref{eq13}--\eqref{eq130} by an application of following Legget-William fixed point theorem. Let $a',b'$ be two real numbers such that $0<a'<b'$ and $\ss$ a nonnegative, continuous, concave functional on $\mathtt{E}.$ We define the following convex sets,
$$\mathtt{E}_{a'}=\{\mathtt{u}\in\mathtt{E}:\Vert \mathtt{u}\Vert<a'\},$$
$$\mathtt{E}(\ss,a',b')=\{\mathtt{u}\in\mathtt{E}:a'\le\ss(\mathtt{u}),\,\Vert \mathtt{u}\Vert<b'\}.$$
\begin{theorem}(Legget-William\cite{legget})\label{t412}
	Let $\mathtt{E}$ be a cone in a Banach space $\mathtt{B}.$ Let $\ss$ a nonnegative, continuous, concave functional on $\mathtt{E}$ satisfying for some $c'>0$ such that $\ss(\mathtt{u})\le\Vert\mathtt{u}\Vert$ for all $\mathtt{u}\in \overline{\mathtt{E}}_{c'}.$ Suppose there exists a completely continuous operator $\aleph :\overline{\mathtt{E}}_{c'}\to\overline{\mathtt{E}}_{c'}$ and $0<a'<b'<d'\le c'$ such that
	\begin{itemize}
		\item [$(a)$] $\{\mathtt{u}\in\mathtt{E}(\ss,b',d'):\ss(\mathtt{u})>a'\}\ne\emptyset$ and $\ss(\aleph \mathtt{u})>b'$ for $\mathtt{u}\in\mathtt{E}(\ss,b',d'),$
		\item [$(b)$] $\Vert\aleph \mathtt{u}\Vert<a'$ for $\Vert\mathtt{u}\Vert<a',$
		\item [$(c)$] $\ss(\aleph \mathtt{u})>b'$ for $\mathtt{u}\in\mathtt{E}(\ss,a',c'),$ with $\Vert\aleph \mathtt{u}\Vert>d'$
	\end{itemize}	
	Then, $\aleph $ has at least three fixed points ${}^1\mathtt{u},{}^2\mathtt{u},{}^3\mathtt{u}\in\mathtt{E}_{c'}$ satisfying
	$\Vert{}^1\mathtt{u}\Vert<a',$ $b'<\ss({}^2\mathtt{u})$ and $\Vert{}^3\mathtt{u}\Vert>a'$ and $\ss({}^3\mathtt{u})<b'.$ 	
\end{theorem}	
\begin{theorem}\label{t42}
Assume that	$(\mathcal{J}_1)$--$(\mathcal{J}_3)$ hold. Let $0<a'<b'<c'$ and suppose that $\mathtt{g}_{\dot{\iota}} ,\,{\dot{\iota}} =1,2,\cdot\cdot\cdot,\mathtt{n}$ satisfies the following conditions,
\begin{itemize}
	\item [$(\mathcal{J}_{11})$] $\mathtt{g}_{\dot{\iota}} (\mathtt{u})<\frac{a'}{\mathfrak{O}_1}$ for $0\le\mathtt{u}\le a',$ where $\mathfrak{O}_1=\frac{ r_0^2}{(\mathtt{N}-2)^2}\Vert\widehat{\Xi}_{r_0}\Vert_\mathtt{q}\prod_{i=1}^{m}\Vert\ell_i\Vert_{\mathtt{p}_i}.$
	\item [$(\mathcal{J}_{12})$] $\mathtt{g}_{\dot{\iota}} (\mathtt{u})>\frac{b'}{\mathfrak{O}_2}$ for $b'\le\mathtt{u}\le c',$ where $\mathfrak{O}_2=\frac{\wp r_0^2}{(\mathtt{N}-2)^2}\prod_{i=1}^{m}\ell_i^\star\int_{0}^{1}\Xi_{r_0}(\mathtt{t}, \mathtt{t})\mathtt{t}^\frac{2(\mathtt{N}-1)}{2-\mathtt{N}} \mathtt{d}\mathtt{t}.$
	\item [$(\mathcal{J}_{13})$] $\mathtt{g}_{\dot{\iota}} (\mathtt{u})<\frac{c'}{\mathfrak{O}_1}$ for $0\le\mathtt{u}\le c'.$	
\end{itemize}
Then the iterative system \eqref{eq13}--\eqref{eq130} has at least three positive solutions $({}^1\mathtt{u}_1,{}^1\mathtt{u}_2,\cdot\cdot\cdot,{}^1\mathtt{u}_\mathtt{n}),$ $({}^2\mathtt{u}_1,{}^2\mathtt{u}_2,\cdot\cdot\cdot,{}^2\mathtt{u}_\mathtt{n})$ and $({}^3\mathtt{u}_1,{}^3\mathtt{u}_2,\cdot\cdot\cdot,{}^3\mathtt{u}_\mathtt{n})$ with $\Vert{}^{\dot{\iota}} \mathtt{u}_1\Vert<a',$ $b'<\ss({}^{\dot{\iota}} \mathtt{u}_2),$ $\Vert{}^{\dot{\iota}} \mathtt{u}_3\Vert>a'$ and $\ss({}^{\dot{\iota}} \mathtt{u}_3)<b'$ for ${\dot{\iota}} =1,2,\cdot\cdot\cdot,\mathtt{n}.$
\end{theorem}
\begin{proof}
From Lemma \ref{l24}, $\aleph :\mathtt{E}\to\mathtt{E}$ is a completely continuous operator. If $\mathtt{u}_1\in\overline{\mathtt{E}}_{c'},$ then $\Vert\mathtt{u}_1\Vert\le c'$ and for $0<\mathtt{t}_{\mathtt{n}-1}<1$ and by $(\mathcal{J}_{13}),$ we have 
$$
\begin{aligned}
	\int_{0}^{1}\Xi_{r_0}(\mathtt{t}_{\mathtt{n}-1}, \mathtt{t}_\mathtt{n})\ell(\mathtt{t}_\mathtt{n})\mathtt{g}_\mathtt{n}\big( \mathtt{u}_1(\mathtt{t}_\mathtt{n})\big)\mathtt{d}\mathtt{t}_\mathtt{n}
	&\le\int_{0}^{1}\Xi_{r_0}(\mathtt{t}_{\mathtt{n}}, \mathtt{t}_\mathtt{n})\ell(\mathtt{t}_\mathtt{n})\mathtt{g}_\mathtt{n}\big( \mathtt{u}_1(\mathtt{t}_\mathtt{n})\big)\mathtt{d}\mathtt{t}_\mathtt{n}\\
	&\le \frac{ c'}{\mathfrak{O}_1}\int_{0}^{1}\Xi_{r_0}(\mathtt{t}_{\mathtt{n}}, \mathtt{t}_\mathtt{n})\ell(\mathtt{t}_\mathtt{n})\mathtt{d}\mathtt{t}_\mathtt{n}\\
	&\le \frac{ c' r_0^2}{(\mathtt{N}-2)^2\mathfrak{O}_1}\int_{0}^{1}\Xi_{r_0}(\mathtt{t}_{\mathtt{n}}, \mathtt{t}_\mathtt{n})\mathtt{t}^\frac{2(\mathtt{N}-1)}{2-\mathtt{N}}_\mathtt{n}\prod_{i=1}^{m}\ell_i(\mathtt{t}_\mathtt{n})\mathtt{d}\mathtt{t}_\mathtt{n}.
\end{aligned}
$$
There exists a $\mathtt{q}>1$ such that $\displaystyle\sum_{i=1}^{m}\frac{1}{\mathtt{p}_i}+\frac{1}{\mathtt{q}}=1.$ By the first part of Theorem \ref{holder}, we have
$$
\begin{aligned}
	\int_{0}^{1}\Xi_{r_0}(\mathtt{t}_{\mathtt{n}-1}, \mathtt{t}_\mathtt{n})\ell(\mathtt{t}_\mathtt{n})\mathtt{g}_\mathtt{n}\big( \mathtt{u}_1(\mathtt{t}_\mathtt{n})\big)\mathtt{d}\mathtt{t}_\mathtt{n}&\le \frac{ c' r_0^2}{(\mathtt{N}-2)^2\mathfrak{O}_1}\Vert\widehat{\Xi}_{r_0}\Vert_\mathtt{q}\prod_{i=1}^{m}\Vert\ell_i\Vert_{\mathtt{p}_i}\\
	&\le c'.
\end{aligned}
$$
Continuing with this bootstrapping argument, we get
$$
\begin{aligned}
	\left\Vert\aleph  \mathtt{u}_1\right\Vert=&\max_{\mathtt{s}\in[0,1]} \int_{0}^{1}\Xi_{r_0}(\mathtt{s}, \mathtt{t}_1)\ell(\mathtt{t}_1)\mathtt{g}_1\Bigg[ \int_{0}^{1}\Xi_{r_0}(\mathtt{t}_1, \mathtt{t}_2)\ell(\mathtt{t}_2)\mathtt{g}_2\Bigg[ \int_{0}^{1}\Xi_{r_0}(\mathtt{t}_2, \mathtt{t}_3)\ell(\mathtt{t}_3)\mathtt{g}_4 \cdots\\
	&\hskip1cm\mathtt{g}_{\mathtt{n}-1}\Bigg[ \int_{0}^{1}\Xi_{r_0}(\mathtt{t}_{\mathtt{n}-1}, \mathtt{t}_\mathtt{n})\ell(\mathtt{t}_\mathtt{n})\mathtt{g}_\mathtt{n}\big( \mathtt{u}_1(\mathtt{t}_\mathtt{n})\big)\mathtt{d}\mathtt{t}_\mathtt{n}\Bigg]\cdot\cdot\cdot \Bigg]\mathtt{d}\mathtt{t}_3\Bigg]\mathtt{d}\mathtt{t}_2\Bigg]\mathtt{d}\mathtt{t}_1\\
	\le&\,c'.
\end{aligned}
$$
Hence, $\aleph :\overline{\mathtt{E}}_{c'}\to\overline{\mathtt{E}}_{c'}.$ In the same way, if $\mathtt{u}_1\in\overline{\mathtt{E}}_{a'},$ then $\aleph :\overline{\mathtt{E}}_{a'}\to\overline{\mathtt{E}}_{a'}.$ Therefore, condition $(b)$ of Theorem \ref{t412} satisfied. To check condition $(a)$ of Theorem \ref{t412}, choose $\mathtt{u}_1(\mathtt{s})=(b'+c')/2,\,\mathtt{s}\in[0,1].$ It is easy to see that $\mathtt{u}_1\in\mathtt{E}(\ss,b',c')$ and $\ss(\mathtt{u}_1)=\ss((b'+c')/2)>b'.$ So, $\{\mathtt{u}_1\in\mathtt{E}(\ss,b',c'):\ss(\mathtt{u}_1)>b'\}\ne\emptyset.$ Hence, if $\mathtt{u}_1\in\mathtt{E}(\ss,b',c')$ then $b'<\mathtt{u}_1(\mathtt{s})<c',\,\mathtt{s}\in[0,1].$ Let $0<\mathtt{t}_{\mathtt{n}-1}<1.$ Then by $(\mathcal{J}_{12}),$ we have
$$
\begin{aligned}
	\int_{0}^{1}\Xi_{r_0}(\mathtt{t}_{\mathtt{n}-1}, \mathtt{t}_\mathtt{n})\ell(\mathtt{t}_\mathtt{n})\mathtt{g}_\mathtt{n}\big( \mathtt{u}_1(\mathtt{t}_\mathtt{n})\big)\mathtt{d}\mathtt{t}_\mathtt{n}
	&\ge \wp\int_{0}^{1}\Xi_{r_0}(\mathtt{t}_{\mathtt{n}}, \mathtt{t}_\mathtt{n})\ell(\mathtt{t}_\mathtt{n})\mathtt{g}_\mathtt{n}\big( \mathtt{u}_1(\mathtt{t}_\mathtt{n})\big)\mathtt{d}\mathtt{t}_\mathtt{n}\\
	&\ge \frac{\wp b'}{\mathfrak{O}_2}\int_{0}^{1}\Xi_{r_0}(\mathtt{t}_{\mathtt{n}}, \mathtt{t}_\mathtt{n})\ell(\mathtt{t}_\mathtt{n})\mathtt{d}\mathtt{t}_\mathtt{n}\\
	&\ge \frac{\wp b'r_0^2}{(\mathtt{N}-2)^2\mathfrak{O}_2}\int_{0}^{1}\Xi_{r_0}(\mathtt{t}_{\mathtt{n}}, \mathtt{t}_\mathtt{n})\mathtt{t}^\frac{2(\mathtt{N}-1)}{2-\mathtt{N}}_\mathtt{n}\prod_{i=1}^{m}\ell_i(\mathtt{t}_\mathtt{n})\mathtt{d}\mathtt{t}_\mathtt{n}\\
	&\ge \frac{\wp b'r_0^2}{(\mathtt{N}-2)^2\mathfrak{O}_2}\prod_{i=1}^{m}\ell_i^\star\int_{0}^{1}\Xi_{r_0}(\mathtt{t}_{\mathtt{n}}, \mathtt{t}_\mathtt{n})\mathtt{t}^\frac{2(\mathtt{N}-1)}{2-\mathtt{N}}_\mathtt{n} \mathtt{d}\mathtt{t}_\mathtt{n}\\
	&\ge b'.
\end{aligned}
$$
Continuing with this bootstrapping argument, we get
$$
\begin{aligned}
	\min_{\mathtt{s}\in[0,1]}\left(\aleph  \mathtt{u}_1\right)=&\min_{\mathtt{s}\in[0,1]} \int_{0}^{1}\Xi_{r_0}(\mathtt{s}, \mathtt{t}_1)\ell(\mathtt{t}_1)\mathtt{g}_1\Bigg[ \int_{0}^{1}\Xi_{r_0}(\mathtt{t}_1, \mathtt{t}_2)\ell(\mathtt{t}_2)\mathtt{g}_2\Bigg[ \int_{0}^{1}\Xi_{r_0}(\mathtt{t}_2, \mathtt{t}_3)\ell(\mathtt{t}_3)\mathtt{g}_4 \cdots\\
	&\hskip1cm\mathtt{g}_{\mathtt{n}-1}\Bigg[ \int_{0}^{1}\Xi_{r_0}(\mathtt{t}_{\mathtt{n}-1}, \mathtt{t}_\mathtt{n})\ell(\mathtt{t}_\mathtt{n})\mathtt{g}_\mathtt{n}\big( \mathtt{u}_1(\mathtt{t}_\mathtt{n})\big)\mathtt{d}\mathtt{t}_\mathtt{n}\Bigg]\cdot\cdot\cdot \Bigg]\mathtt{d}\mathtt{t}_3\Bigg]\mathtt{d}\mathtt{t}_2\Bigg]\mathtt{d}\mathtt{t}_1\\
	\ge&\,b'.
\end{aligned}
$$
Therefore, we have
$$\ss(\aleph \mathtt{u}_1)>b',~\textnormal{for}~\mathtt{u}_1\in\mathtt{E}(\ss,b',c').$$
This implies that condition $(a)$ of Theorem \ref{t412} is satisfied.
 
Finally, if $\mathtt{u}_1\in\mathtt{E}(\ss,b',c'),$ then what we have already proved, $\ss(\aleph \mathtt{u}_1)>b'.$ Which proves the condition $(c)$ of Theorem \ref{t412}. To sum up, all the conditions of Theorem \ref{t412} are satisfied. Therefore, $\aleph $ has at least three fixed points, that is, problem \eqref{eq13}--\eqref{eq130} has at least three positive solutions $({}^1\mathtt{u}_1,{}^1\mathtt{u}_2,\cdot\cdot\cdot,{}^1\mathtt{u}_\mathtt{n}),$ $({}^2\mathtt{u}_1,{}^2\mathtt{u}_2,\cdot\cdot\cdot,{}^2\mathtt{u}_\mathtt{n})$ and $({}^3\mathtt{u}_1,{}^3\mathtt{u}_2,\cdot\cdot\cdot,{}^3\mathtt{u}_\mathtt{n})$ with $\Vert{}^{\dot{\iota}} \mathtt{u}_1\Vert<a',$ $b'<\ss({}^{\dot{\iota}} \mathtt{u}_2),$ $\Vert{}^{\dot{\iota}} \mathtt{u}_3\Vert>a'$ and $\ss({}^{\dot{\iota}} \mathtt{u}_3)<b'$ for ${\dot{\iota}} =1,2,\cdot\cdot\cdot,\mathtt{n}.$ \qed
\end{proof}		
For $\sum_{i=1}^{m}\frac{1}{\mathtt{p}_i}=1$ and $\sum_{i=1}^{m}\frac{1}{\mathtt{p}_i}>1,$ we have following results.
\begin{theorem}
	Assume that	$(\mathcal{J}_1)$--$(\mathcal{J}_3)$ hold. Let $0<a'<b'< c'$ and suppose that $\mathtt{g}_{\dot{\iota}} ,\,{\dot{\iota}} =1,2,\cdot\cdot\cdot,\mathtt{n}$ satisfies $(\mathcal{J}_{12}),\,(\mathcal{J}_{13})$ and 
	\begin{itemize}
		\item [$(\mathcal{J}_{14})$] $\mathtt{g}_{\dot{\iota}} (\mathtt{u})<\frac{a'}{\mathfrak{O}_3}$ for $0\le\mathtt{u}\le a',$ where $\mathfrak{O}_3=\frac{ r_0^2}{(\mathtt{N}-2)^2}\Vert\widehat{\Xi}_{r_0}\Vert_{\infty}\prod_{i=1}^{m}\Vert\ell_i\Vert_{\mathtt{p}_i}.$
	\end{itemize}
	Then the iterative system \eqref{eq13}--\eqref{eq130} has at least three positive solutions $({}^1\mathtt{u}_1,{}^1\mathtt{u}_2,\cdot\cdot\cdot,{}^1\mathtt{u}_\mathtt{n}),$ $({}^2\mathtt{u}_1,{}^2\mathtt{u}_2,\cdot\cdot\cdot,{}^2\mathtt{u}_\mathtt{n})$ and $({}^3\mathtt{u}_1,{}^3\mathtt{u}_2,\cdot\cdot\cdot,{}^3\mathtt{u}_\mathtt{n})$ with $\Vert{}^{\dot{\iota}} \mathtt{u}_1\Vert<a',$ $b'<\ss({}^{\dot{\iota}} \mathtt{u}_2),$ $\Vert{}^{\dot{\iota}} \mathtt{u}_3\Vert>a'$ and $\ss({}^{\dot{\iota}} \mathtt{u}_3)<b'$ for ${\dot{\iota}} =1,2,\cdot\cdot\cdot,\mathtt{n}.$
\end{theorem}
\begin{theorem}
	Assume that	$(\mathcal{J}_1)$--$(\mathcal{J}_3)$ hold. Let $0<a'<b'< c'$ and suppose that $\mathtt{g}_{\dot{\iota}} ,\,{\dot{\iota}} =1,2,\cdot\cdot\cdot,\mathtt{n}$ satisfies $(\mathcal{J}_{12}),\,(\mathcal{J}_{13})$ and 
	\begin{itemize}
		\item [$(\mathcal{J}_{15})$] $\mathtt{g}_{\dot{\iota}} (\mathtt{u})<\frac{a'}{\mathfrak{O}_4}$ for $0\le\mathtt{u}\le a',$ where $\mathfrak{O}_4=\frac{ r_0^2}{(\mathtt{N}-2)^2}\Vert\widehat{\Xi}_{r_0}\Vert_{\infty}\prod_{i=1}^{m}\Vert\ell_i\Vert_1.$
	\end{itemize}
	Then the iterative system \eqref{eq13}--\eqref{eq130} has at least three positive solutions $({}^1\mathtt{u}_1,{}^1\mathtt{u}_2,\cdot\cdot\cdot,{}^1\mathtt{u}_\mathtt{n}),$ $({}^2\mathtt{u}_1,{}^2\mathtt{u}_2,\cdot\cdot\cdot,{}^2\mathtt{u}_\mathtt{n})$ and $({}^3\mathtt{u}_1,{}^3\mathtt{u}_2,\cdot\cdot\cdot,{}^3\mathtt{u}_\mathtt{n})$ with $\Vert{}^{\dot{\iota}} \mathtt{u}_1\Vert<a',$ $b'<\ss({}^{\dot{\iota}} \mathtt{u}_2),$ $\Vert{}^{\dot{\iota}} \mathtt{u}_3\Vert>a'$ and $\ss({}^{\dot{\iota}} \mathtt{u}_3)<b'$ for ${\dot{\iota}} =1,2,\cdot\cdot\cdot,\mathtt{n}.$
\end{theorem}
\begin{example}
	Consider the following nonlinear elliptic system of equations,
\begin{equation}\label{e8}
	\triangle{  \mathtt{u}_{{\dot{\iota}} }}+\frac{(\mathtt{N}-2)^2r_0^{2\mathtt{N}-2}}{\vert x\vert^{2\mathtt{N}-2}}\mathtt{u}_{\dot{\iota}} +\ell(\vert x\vert)\mathtt{g}_{{\dot{\iota}} }(\mathtt{u}_{{\dot{\iota}} +1})=0,~1<\vert x\vert<2,
\end{equation}
\begin{equation}\label{e9}\left.
	\begin{aligned}
		&\hskip0.9cm \mathtt{u}_{{\dot{\iota}} }=0~~\text{on}~~\vert x\vert=1~\text{and}~\vert x\vert=2,\\
		&\mathtt{u}_{{\dot{\iota}} }=0~~\text{on}~~\vert x\vert=1~\text{and}~\frac{\partial \mathtt{u}_{\dot{\iota}} }{\partial r}=0~\text{on}~\vert x\vert=2,\\
		&\frac{\partial \mathtt{u}_{\dot{\iota}} }{\partial r}=0~~\text{on}~~\vert x\vert=1~\text{and}~\mathtt{u}_{{\dot{\iota}} }=0~\text{on}~\vert x\vert=2,
	\end{aligned}\right\}
\end{equation}
\end{example}
where $ r_0=1,$ $\mathtt{N}=3,$ ${\dot{\iota}} \in\{1,2\},\, \mathtt{u}_3= \mathtt{u}_1,$  $\ell(\mathtt{s})=\frac{1}{\mathtt{s}^4}\prod_{i=1}^{2}\ell_i(\mathtt{s}),$ $\ell_i(\mathtt{s})=\ell_i\left(\frac{1}{\mathtt{s}}\right),$
in which
$$\ell_1(t)=\frac{1}{\sqrt{t+1}}~~~~\text{~~and~~}~~~~\ell_2(t)=\frac{1}{\sqrt{t^2+25}},$$
then it is clear that
$$\ell_1, \ell_2\in L^\mathtt{p}[0,1]~~\text{~~and~~}\prod_{i=1}^{2}\ell_i^*=5.$$
Let 
$$
\begin{aligned}
\mathtt{g}_1(\mathtt{u})=\mathtt{g}_2(\mathtt{u})=\left\{\begin{array}{ll}\frac{3}{2}, \hskip1.3cm  \mathtt{u}\ge1,\vspace{1.2mm}\\	
\frac{1}{2}\mathtt{u}^2+1, \hskip0.4cm  \mathtt{u}<1.
\end{array}
\right.
\end{aligned}
$$
Let $\upalpha=\upbeta=\gamma=\updelta=1,$  $\varrho=2\cosh(1)+2\sinh(1)\approx5.436563658,$ 
$$
\begin{aligned}
\Xi_{r_0}(\mathtt{s},\mathtt{t})=\frac{1}{2\cosh(1)+2\sinh(1)}\left\{\begin{array}{ll}\big(\sinh(\mathtt{s})+\cosh(\mathtt{s})\big)\big(\sinh(1-\mathtt{t})+\cosh(1-\mathtt{t})\big), \hskip0.6cm 0  \le \mathtt{s}\le \mathtt{t}\le 1,\vspace{1.2mm}\\	
	\big(\sinh(\mathtt{t})+\cosh(\mathtt{t})\big)\big(\sinh(1-\mathtt{s})+\cosh(1-\mathtt{s})\big),
	\hskip0.6cm  0  \le \mathtt{t}\le \mathtt{s}\le 1,
\end{array}
\right.
\end{aligned}
$$
and $\wp=\frac{3}{\cosh(1)}.$ 
Also, 
$$
\begin{aligned}
\mathfrak{O}_1=\frac{\wp r_0^2}{(\mathtt{N}-2)^2}\prod_{i=1}^{m}\ell_i^\star\int_{0}^{1}\Xi_{r_0}(\mathtt{t}, \mathtt{t})\mathtt{t}^\frac{2(\mathtt{N}-1)}{2-\mathtt{N}} \mathtt{d}\mathtt{t}\approx4.627034665\times10^6.
\end{aligned}
$$
Let $\mathtt{p}_1=3,\mathtt{p}_2=2$ and $\mathtt{q}=6,$ then $\frac{1}{\mathtt{p}_1}+\frac{1}{\mathtt{p}_2}+\frac{1}{\mathtt{q}}=1$ and
$$
\begin{aligned}
\mathfrak{O}_2=\frac{  r_0^2}{(\mathtt{N}-2)^2}\Vert\widehat{\Xi}_{r_0}\Vert_\mathtt{q}\prod_{i=1}^{m}\Vert\ell_i\Vert_{\mathtt{p}_i}\approx9.696074194\times10^{7}.
\end{aligned}
$$
Choose $a'=10^{7},\,b'=10^{8}$ and $c'=10^{9}.$ Then,
$$
\begin{aligned}
\mathtt{g}_1(\mathtt{u})=\mathtt{g}_2(\mathtt{u})&\,\le2.161211386=\frac{a'}{\mathfrak{O}_1},~\mathtt{u}\in[0,10^{7}],\\
\mathtt{g}_1(\mathtt{u})=\mathtt{g}_2(\mathtt{u})&\,\ge1.031345243=\frac{b'}{\mathfrak{O}_2},~\mathtt{u}\in[10^{8},10^{9}],\\
\mathtt{g}_1(\mathtt{u})=\mathtt{g}_2(\mathtt{u})&\,\le216.1211386=\frac{c'}{\mathfrak{O}_1},~\mathtt{u}\in[0,10^{9}].
\end{aligned}
$$
Therefore, by Theorem \ref{t43}, the boundary value problem \eqref{e6}--\eqref{e7} has at least two positive radial solutions $({}^{\dot{\iota}} \mathtt{u}_1,{}^{\dot{\iota}} \mathtt{u}_2),\,{\dot{\iota}} =1,2$ such that 
$$\max_{\mathtt{s}\in[0,1]}{}^{\dot{\iota}} \mathtt{u}_1(\mathtt{s})<10^{7},~10^{8}<\min_{\mathtt{s}\in[0,1]}{}^{\dot{\iota}} \mathtt{u}_2(\mathtt{s})<\max_{\mathtt{s}\in[0,1]}{}^{\dot{\iota}} \mathtt{u}_2(\mathtt{s})<10^{9},~~\text{for}~~{\dot{\iota}} =1,2,$$
$$10^{7}<\max_{\mathtt{s}\in[0,1]}{}^{\dot{\iota}} \mathtt{u}_3(\mathtt{s})<10^{9},~~\min_{\mathtt{s}\in[0,1]}{}^{\dot{\iota}} \mathtt{u}_3(\mathtt{s})<10^{8},~~\text{for}~~{\dot{\iota}} =1,2.$$	
\section{Existence of Unique Positive Radial Solution}
In the next, for the existence of unique positive radial solution to the boundary value problem \eqref{eq13}--\eqref{eq130}, we employ two metrics under Rus’s theorem (see, \cite{alm, stin} for more details). Consider the set of real valued functions that are defined and continuous on $[ 0, 1]$ and denote this space by $\mathtt{v}= C([ 0, 1]).$ For functions $\mathtt{v}_1,\mathtt{v}_2\in\mathtt{v},$ consider the following two metrics on $\mathtt{v}:$
\begin{equation}\label{eq311}
	\mathtt{d}(\mathtt{v}_1,\mathtt{v}_2)=\max_{\mathtt{t}\in[ 0, 1]}\vert\mathtt{v}_1(\mathtt{t})-\mathtt{v}_2(\mathtt{t})\vert,
\end{equation}
\begin{equation}\label{eq312}
	\uprho(\mathtt{v}_1,\mathtt{v}_2)=\left[\int_{ 0}^{1}\vert\mathtt{v}_1(\mathtt{t})-\mathtt{v}_2(\mathtt{t})\vert^{\mathtt{p}}\mathtt{d}\mathtt{t}\right]^{\frac{1}{\mathtt{p}}},~~\mathtt{p}>1.
\end{equation}
For $\mathtt{d}$ in \eqref{eq311}, the pair $(C([ 0, 1]), \mathtt{d})$ forms a complete metric space. For $\uprho$ in  \eqref{eq312}, the pair $(C([ 0, 1]), \uprho)$ forms a metric space. The relationship between the two metrics on $\mathtt{v}$ is given by
\begin{equation}\label{eq313}
	\uprho(\mathtt{v}_1,\mathtt{v}_2)\le\mathtt{d}(\mathtt{v}_1,\mathtt{v}_2)~~\text{for~all}~~\mathtt{v}_1,\mathtt{v}_2\in\mathtt{v}.
\end{equation}
\begin{theorem}[Rus \cite{rus}]\label{t2}
	Let $\mathtt{v}$ be a nonempty set and let $\mathtt{d}$ and $\uprho$ be two
	metrics on $\mathtt{v}$ such that $(\mathtt{v}, \mathtt{d})$ forms a complete metric space. If the mapping $\Lambda :\mathtt{v}\to \mathtt{v}$ is continuous with respect to $\mathtt{d}$ on $\mathtt{v}$ and 
	\begin{equation}\label{eq31}
		\mathtt{d}(\Lambda \mathtt{v}_1,\Lambda \mathtt{v}_2)\le\alpha_1\uprho(\mathtt{v}_1,\mathtt{v}_2),
	\end{equation}
	for some $\alpha_1>0$ and for all $\mathtt{v}_1,\mathtt{v}_2\in\mathtt{v},$
	\begin{equation}\label{eq32}
		\uprho(\Lambda \mathtt{v}_1,\Lambda \mathtt{v}_2)\le\alpha_2\uprho(\mathtt{v}_1,\mathtt{v}_2),
	\end{equation}
	for some $0<\alpha_2<1$ for all $\mathtt{v}_1,\mathtt{v}_2\in\mathtt{v},$
	then there is a unique $\mathtt{v}^*\in\mathtt{v}$ such that $\Lambda \mathtt{v}^*=\mathtt{v}^*.$
\end{theorem}
Denote $\Upsilon(\mathtt{t})=\Xi_{r_0}(\mathtt{t}, \mathtt{t})\mathtt{t}^\frac{2(\mathtt{N}-1)}{2-\mathtt{N}}\prod_{i=1}^{m}\ell_i(\mathtt{t})\mathtt{d}\mathtt{t}.$	
\begin{theorem}\label{tmet}
	Assume that $(\mathcal{J}_1),\,(\mathcal{J}_3)$ and the following codition are satisfied.
	\begin{itemize}
		\item [$(\mathcal{J}_{14})$] there exists a number $\mathtt{K}>0$ such that
		$$\vert\mathtt{g}_{\dot{\iota}} (\mathtt{u})-\mathtt{g}_{\dot{\iota}} (\mathtt{v})\vert\le\mathtt{K}\vert\mathtt{u}-\mathtt{v}\vert~~\text{for}~~\mathtt{u},\mathtt{v}\in\mathtt{v}.$$
	\end{itemize} Further, assume that there are constants $\mathtt{p}>1$ and $\mathtt{q}>1$ such that $1/\mathtt{p}+1/\mathtt{q}=1$ with
	\begin{equation}\label{eq81}
		\left[\frac{\wp\mathtt{K}r_0^2}{(\mathtt{N}-2)^2}\right]^{\mathtt{n}+1}\left[\int_0^1\vert\Upsilon(\mathtt{t})\vert \mathtt{d}\mathtt{t}\right]^\mathtt{n}\left[\int_0^1\vert\Upsilon(\mathtt{t})\vert^\mathtt{q} \mathtt{d}\mathtt{t}\right]^{\frac{1}{q}}<1,
	\end{equation}
	then the boundary value problem \eqref{eq13}--\eqref{eq130} has a unique positive radial solution in $\mathtt{v}.$ 
\end{theorem}
\begin{proof}
	Let $\mathtt{u}_1, \mathtt{v}_1\in C([ 0, 1])$ and $\mathtt{t}\in[ 0, 1].$ Then by H\"older's inequality, we have 
	\begin{align}
		\bigg\vert\int_{0}^{1}\Xi_{r_0}(\mathtt{t}_{\mathtt{n}-1}, \mathtt{t}_\mathtt{n})&\ell(\mathtt{t}_\mathtt{n})\mathtt{g}_\mathtt{n}\big( \mathtt{u}_1(\mathtt{t}_\mathtt{n})\big)\mathtt{d}\mathtt{t}_\mathtt{n}-\int_{0}^{1}\Xi_{r_0}(\mathtt{t}_{\mathtt{n}-1}, \mathtt{t}_\mathtt{n})\ell(\mathtt{t}_\mathtt{n})\mathtt{g}_\mathtt{n}\big( \mathtt{v}_1(\mathtt{t}_\mathtt{n})\big)\mathtt{d}\mathtt{t}_\mathtt{n}\bigg\vert\nonumber\\
		&\le\int_{0}^{1}\vert\Xi_{r_0}(\mathtt{t}_{\mathtt{n}-1}, \mathtt{t}_\mathtt{n})\,\ell(\mathtt{t}_\mathtt{n})\vert\vert\mathtt{g}_\mathtt{n}\big( \mathtt{u}_1(\mathtt{t}_\mathtt{n})\big)-\mathtt{g}_\mathtt{n}\big( \mathtt{v}_1(\mathtt{t}_\mathtt{n})\big)\vert \mathtt{d}\mathtt{t}_\mathtt{n}\nonumber\\
		&\le\int_{0}^{1}\vert\Xi_{r_0}(\mathtt{t}_{\mathtt{n}}, \mathtt{t}_\mathtt{n})\,\ell(\mathtt{t}_\mathtt{n})\vert\,\mathtt{K}\vert \mathtt{u}_1(\mathtt{t}_\mathtt{n})-\mathtt{v}_1(\mathtt{t}_\mathtt{n})\vert \mathtt{d}\mathtt{t}_\mathtt{n}\nonumber\\
		&\le \frac{\mathtt{K} r_0^2}{(\mathtt{N}-2)^2}\int_{0}^{1}\vert\Upsilon(\mathtt{t}_\mathtt{n})\vert\vert \mathtt{u}_1(\mathtt{t}_\mathtt{n})-\mathtt{v}_1(\mathtt{t}_\mathtt{n})\vert \mathtt{d}\mathtt{t}_\mathtt{n}\nonumber\\
		&\le \frac{\mathtt{K} r_0^2}{(\mathtt{N}-2)^2}\left[\int_{0}^{1}\vert\Upsilon(\mathtt{t}_\mathtt{n})\vert^\mathtt{q}\mathtt{d}\mathtt{t}_\mathtt{n}\right]^\frac{1}{\mathtt{q}}\left[\int_0^1\vert \mathtt{u}_1(\mathtt{t}_\mathtt{n})-\mathtt{v}_1(\mathtt{t}_\mathtt{n})\vert^\mathtt{p} \mathtt{d}\mathtt{t}_\mathtt{n}\right]^\frac{1}{\mathtt{p}}\nonumber\\
		&\le \frac{\mathtt{K} r_0^2}{(\mathtt{N}-2)^2}\left[\int_{0}^{1}\vert\Upsilon(\mathtt{t}_\mathtt{n})\vert^\mathtt{q}\mathtt{d}\mathtt{t}_\mathtt{n}\right]^\frac{1}{\mathtt{q}}\uprho(\mathtt{u}_1,\mathtt{v}_1)\nonumber\\
		&\le\alpha_1^\star\uprho(\mathtt{u}_1,\mathtt{v}_1),\nonumber
	\end{align}
	where
	$$\alpha_1^\star=\frac{\mathtt{K} r_0^2}{(\mathtt{N}-2)^2}\left[\int_{0}^{1}\vert\Upsilon(\mathtt{t}_\mathtt{n})\vert^\mathtt{q}\right]^\frac{1}{\mathtt{q}}.$$
	It follows in similar manner for $0<\mathtt{t}_{\mathtt{n}-2}<1,$
	$$
	\begin{aligned}
		\bigg\vert\int_{0}^{1}\Xi_{r_0}(&\mathtt{t}_{\mathtt{n}-2}, \mathtt{t}_{\mathtt{n}-1})\ell(\mathtt{t}_{\mathtt{n}-1})\mathtt{g}_{\mathtt{n}-1}\bigg[ \int_{0}^{1}\Xi_{r_0}(\mathtt{t}_{\mathtt{n}-1}, \mathtt{t}_\mathtt{n})\ell(\mathtt{t}_\mathtt{n})\mathtt{g}_\mathtt{n}\big( \mathtt{u}_1(\mathtt{t}_\mathtt{n})\big)\mathtt{d}\mathtt{t}_\mathtt{n}\bigg]\mathtt{d}\mathtt{t}_{\mathtt{n}-1}\\
		&\hskip1cm-\int_{0}^{1}\Xi_{r_0}(\mathtt{t}_{\mathtt{n}-2}, \mathtt{t}_{\mathtt{n}-1})\ell(\mathtt{t}_{\mathtt{n}-1})\mathtt{g}_{\mathtt{n}-1}\bigg[ \int_{0}^{1}\Xi_{r_0}(\mathtt{t}_{\mathtt{n}-1}, \mathtt{t}_\mathtt{n})\ell(\mathtt{t}_\mathtt{n})\mathtt{g}_\mathtt{n}\big( \mathtt{u}_1(\mathtt{t}_\mathtt{n})\big)\mathtt{d}\mathtt{t}_\mathtt{n}\bigg]\mathtt{d}\mathtt{t}_{\mathtt{n}-1}\bigg\vert\\
		&\le\frac{\mathtt{K} r_0^2}{(\mathtt{N}-2)^2}\int_{0}^{1}\vert\Upsilon(\mathtt{t}_{\mathtt{n}-1})\vert\alpha_1\uprho(\mathtt{u}_1,\mathtt{v}_1)\mathtt{d}\mathtt{t}_{\mathtt{n}-1}\\
		&\le \widehat{\alpha}_1{\alpha}_1^\star\uprho(\mathtt{u}_1,\mathtt{v}_1),
	\end{aligned}
	$$
	where $$\widehat{\alpha}_1=\frac{\mathtt{K} r_0^2}{(\mathtt{N}-2)^2}\int_{0}^{1}\vert\Upsilon(\mathtt{t})\vert \mathtt{d}\mathtt{t}.$$
	Continuing with bootstrapping argument, we get
	$$\left\vert\Lambda\mathtt{u}_1(\mathtt{t})-\Lambda\mathtt{v}_1(\mathtt{t})\right\vert\le\widehat{\alpha}_1^{\mathtt{n}}\alpha_1^\star\uprho(\mathtt{u}_1,\mathtt{v}_1).$$
	we see that
	\begin{equation}\label{eq51}
		\mathtt{d}(\Lambda \mathtt{u}_1,\Lambda \mathtt{v}_1)\le\alpha_1\uprho(\mathtt{u}_1,\mathtt{v}_1),
	\end{equation}
	for some $\alpha_1=\widehat{\mathtt{c}}_1^{\mathtt{n}}\alpha_1^\star>0$ for all $\mathtt{u}_1,\mathtt{v}_1\in\mathtt{v},$
	and so the inequality \eqref{eq31} of Theorem \ref{t2} holds. Now, for all $\mathtt{u}_1,\mathtt{v}_1\in\mathtt{v},$ we may apply \eqref{eq313} to \eqref{eq51} to obtain
	$$\mathtt{d}(\Lambda \mathtt{u}_1,\Lambda \mathtt{v}_1)\le\alpha_1\uprho(\mathtt{u}_1,\mathtt{v}_1)\le\alpha_1\mathtt{d}(\mathtt{u}_1,\mathtt{v}_1).$$
	Thus, given any $\varepsilon>0$ we can choose $\upeta=\varepsilon/\alpha_1$ so that $\mathtt{d}(\Lambda \mathtt{u}_1,\Lambda \mathtt{v}_1)<\varepsilon,$
	whenever $\mathtt{d}(\mathtt{u}_1,\mathtt{v}_1)<\upeta.$ Hence $\Lambda $ is continuous on $\mathtt{v}$ with respect to the metric $\mathtt{d}.$ Finally, we show that $\Lambda $ is contractive on $\mathtt{v}$ with respect to the metric $\uprho.$ From \eqref{eq51}, for each $\mathtt{u}_1,\mathtt{v}_1\in\mathtt{v}$ consider
	$$
	\begin{aligned}
		\bigg[\int_{0}^{1}\vert(\Lambda\mathtt{u}_1)(\mathtt{t})-(\Lambda \mathtt{v}_1)(\mathtt{t})\vert^\mathtt{p}\mathtt{d}\mathtt{t}\bigg]^{\frac{1}{\mathtt{p}}}
		&\le\left[\int_{ 0}^{1}\left\vert\widehat{\mathtt{c}}_1^{\mathtt{n}}\alpha_1^\star\uprho(\mathtt{u}_1,\mathtt{v}_1)\right\vert^\mathtt{p}\mathtt{d}\mathtt{t}\right]^{\frac{1}{\mathtt{p}}}\\
		&\le\,\left[\frac{\mathtt{K}r_0^2}{(\mathtt{N}-2)^2}\right]^{\mathtt{n}+1}\left[\int_0^1\vert\Upsilon(\mathtt{t})\vert \mathtt{d}\mathtt{t}\right]^\mathtt{n}\left[\int_0^1\vert\Upsilon(\mathtt{t})\vert^\mathtt{q} \mathtt{d}\mathtt{t}\right]^{\frac{1}{q}}\uprho(\mathtt{v}_1,\mathtt{v}_2).
	\end{aligned}
	$$
	That is $$
	\begin{aligned}
		\uprho(\Lambda \mathtt{u}_1,\Lambda \mathtt{v}_1)\le\left[\frac{\mathtt{K}r_0^2}{(\mathtt{N}-2)^2}\right]^{\mathtt{n}+1}\left[\int_0^1\vert\Upsilon(\mathtt{t})\vert \mathtt{d}\mathtt{t}\right]^\mathtt{n}\left[\int_0^1\vert\Upsilon(\mathtt{t})\vert^\mathtt{q} \mathtt{d}\mathtt{t}\right]^{\frac{1}{q}}\uprho(\mathtt{v}_1,\mathtt{v}_2).
	\end{aligned}
	$$
	From the assumption \eqref{eq81}, we have
	$$\uprho(\Lambda \mathtt{v}_1,\Lambda \mathtt{v}_2)\le\alpha_2\uprho(\mathtt{v}_1,\mathtt{v}_2)$$
	for some $\alpha_2<1$ and all $\mathtt{v}_1,\mathtt{v}_2\in\mathtt{v}.$ Thus, Theorem \ref{t2}, the operator $\Lambda $ has a unique fixed point in $\mathtt{v}.$ Also, we note that the operator $\Lambda$ is positive from Lemma \ref{l24}. Therefore, the boundary value problem \eqref{eq11} has a unique positive radial solution.\qed
\end{proof}	
\begin{example}
	Consider the following nonlinear elliptic system of equations,
	\begin{equation}\label{e10}
		\triangle{  \mathtt{u}_{{\dot{\iota}} }}+\frac{(\mathtt{N}-2)^2r_0^{2\mathtt{N}-2}}{\vert x\vert^{2\mathtt{N}-2}}\mathtt{u}_{\dot{\iota}} +\ell(\vert x\vert)\mathtt{g}_{{\dot{\iota}} }(\mathtt{u}_{{\dot{\iota}} +1})=0,~1<\vert x\vert<2,
	\end{equation}
	\begin{equation}\label{e11}\left.
		\begin{aligned}
			&\hskip0.9cm \mathtt{u}_{{\dot{\iota}} }=0~~\text{on}~~\vert x\vert=1~\text{and}~\vert x\vert=2,\\
			&\mathtt{u}_{{\dot{\iota}} }=0~~\text{on}~~\vert x\vert=1~\text{and}~\frac{\partial \mathtt{u}_{\dot{\iota}} }{\partial r}=0~\text{on}~\vert x\vert=2,\\
			&\frac{\partial \mathtt{u}_{\dot{\iota}} }{\partial r}=0~~\text{on}~~\vert x\vert=1~\text{and}~\mathtt{u}_{{\dot{\iota}} }=0~\text{on}~\vert x\vert=2,
		\end{aligned}\right\}
	\end{equation}
\end{example}
where $ r_0=1,$ $\mathtt{N}=3,$ ${\dot{\iota}} \in\{1,2\},\, \mathtt{u}_3= \mathtt{u}_1,$  $\ell(\mathtt{s})=\frac{1}{\mathtt{s}^4}\prod_{i=1}^{2}\ell_i(\mathtt{s}),$ $\ell_i(\mathtt{s})=\ell_i\left(\frac{1}{\mathtt{s}}\right),$
in which
$\ell_1(t)=\ell_2(t)=\frac{1}{t+1},$
then $\prod_{i=1}^{2}\ell_i^*=1.$
Let $\mathtt{g}_1(\mathtt{u})=\frac{1}{10^{4}}\cos(\mathtt{u}),\,\mathtt{g}_2(\mathtt{u})=\frac{\mathtt{u}}{10^{4}(\mathtt{u}+1)}$ and $\upalpha=\upbeta=\upgamma=\updelta=1,$ then  $\uprho=2\cosh(1)+2\sinh(1)\approx5.436563658,$ 
$$
\begin{aligned}
	\Xi_{r_0}(\mathtt{s},\mathtt{t})=\frac{1}{2\cosh(1)+2\sinh(1)}\left\{\begin{array}{ll}\big(\sinh(\mathtt{s})+\cosh(\mathtt{s})\big)\big(\sinh(1-\mathtt{t})+\cosh(1-\mathtt{t})\big), \hskip0.6cm 0  \le \mathtt{s}\le \mathtt{t}\le 1,\vspace{1.2mm}\\	
		\big(\sinh(\mathtt{t})+\cosh(\mathtt{t})\big)\big(\sinh(1-\mathtt{s})+\cosh(1-\mathtt{s})\big),
		\hskip0.6cm  0  \le \mathtt{t}\le \mathtt{s}\le 1,
	\end{array}
	\right.
\end{aligned}
$$
and $\wp=\frac{3}{\cosh(1)}.$ Then,
$$\vert\mathtt{g}_1(\mathtt{u})-\mathtt{g}_1(\mathtt{v})\vert=\frac{\vert\cos(\mathtt{u})-\cos(\mathtt{v})\vert}{10^{4}}\le\frac{1}{10^{4}}\vert\mathtt{u}-\mathtt{v}\vert,$$
and
$$\vert\mathtt{g}_2(\mathtt{u})-\mathtt{g}_2(\mathtt{v})\vert=\frac{1}{10^{4}}\left\vert\frac{\mathtt{u}}{\mathtt{u}+1}-\frac{\mathtt{v}}{\mathtt{v}+1}\right\vert\le\frac{1}{10^{4}}\vert\mathtt{u}-\mathtt{v}\vert.$$
So, $\mathtt{K}=\frac{1}{10^{4}}.$ Let $\mathtt{n}=2$ and $\mathtt{p}=\mathtt{q}=2.$ Then,
$$
\left[\frac{\mathtt{K}r_0^2}{(\mathtt{N}-2)^2}\right]^{\mathtt{n}+1}\left[\int_0^1\vert\Upsilon(\mathtt{t})\vert \mathtt{d}\mathtt{t}\right]^\mathtt{n}\left[\int_0^1\vert\Upsilon(\mathtt{t})\vert^\mathtt{q} \mathtt{d}\mathtt{t}\right]^{\frac{1}{q}}\approx0.3149700790<1.
$$
Therefore, from Theorem \ref{tmet}, the iterative system of boundary value problems \eqref{e10}--\eqref{e11} has a unique positive radial solution.\\

\noindent\textbf{Acknowledgement} \\

\noindent\textbf{Author contributions} The study was carried out in collaboration of all authors. All authors read and approved the final manuscript.\\

\noindent\textbf{Funding} Not Applicable.\\

\noindent\textbf{Data availibility statement} Data sharing not applicable to this paper as no data sets were generated or analyzed during the current study.
\section*{Compliance with ethical standards}
\textbf{Conflict of interest} It is declared that authors has no competing interests.\\

\noindent\textbf{Ethical approval} This article does not contain any studies with human participants or animals performed by any of the authors.


\begin{thebibliography}{99}
\bibitem{ali3}J. Ali and S. Padhi, Existence of multiple positive radial solutions to elliptic equations in an annulus, {\it Com. App. Anal.} {\bf 22}(4), 695--710 (2018).

\bibitem{alm}S. S. Almuthaybiria, C. C. Tisdell,  Sharper existence and uniqueness results for solutions to third order boundary value problems, {\it Math. Model. Anal.}, {\bf25}(3) (2020), 409--420.
	
\bibitem{avery} R. I. Avery, J. Henderson, \textit{Two positive fixed points of nonlinear operators on ordered Banach spaces}, Comm. Appl. Nonlinear Anal. 8 (2001) 27--36.	
	
\bibitem{dal} R. Dalmasso, Existence and uniqueness of positive solutions of semilinear elliptic systems, {\it Nonlinear Anal.} \textbf{39}\,(5) (2000) 559--568.

\bibitem{hai1} D.D. Hai, R. Shivaji, An existence result on positive solutions for a class of semilinear elliptic systems, {\it Proc. R. Soc. Edinb.}, Sect. A \textbf{134}\,(1) (2004) 137--141.

\bibitem{ali1}J. Ali, R. Shivaji, M. Ramaswamy, Multiple positive solutions for classes of elliptic systems with combined nonlinear effects, {\it Differ. Integral Equ.} \textbf{19}\,(6) (2006) 669--680.

\bibitem{hai2}D.D. Hai, Uniqueness of positive solutions for semilinear elliptic systems, {\it J. Math. Anal. Appl.} \textbf{313}\,(2) (2006) 761--767.

\bibitem{hai3}D.D. Hai, R. Shivaji, Uniqueness of positive solutions for a class of semipositone elliptic systems, {\it Nonlinear Anal.} \textbf{66}\,(2) (2007) 396--402.

\bibitem{ali2}J. Ali, K. Brown, R. Shivaji, Positive solutions for $n\times n$ elliptic systems with combined nonlinear effects, {\it Differ. Integral Equ.} \textbf{24}\,(3-4) (2011) 307--324.

\bibitem{chro}M. B. Chrouda, K. Hassine, Uniqueness of positive radial solutions for elliptic equations in an annulus, {\it Proc. Amer. Math. Soc.}, (2020). https://doi.org/10.1090/proc/15286


\bibitem{guo}D. Guo and V. Lakshmikantham: \textit{ Nonlinear Problems in Abstract
	Cones}, Academic Press, San Diego, (1988).

\bibitem{kaji}R. Kajikiya, E. Ko, Existence of positive radial solutions for a semipositone elliptic equation, {\it J. Math. Anal. Appl.} \textbf{484} (2020) 123735.

\bibitem{lan} K. Lan and J. R. L. Webb, Positive solutions of semilinear differential equations with singularities, {\em J. Differ. Equ.} \textbf{148}, 407--421 (1998).

\bibitem{lee}Y. H. Lee, Multiplicity of positive radial solutions for multiparameter semilinear elliptic systems on an annulus. {\em J. of Differ. Equ.} {\bf 174}(2), 420--441 (2001).

\bibitem{legget}R. W. Legget, L. R. Williams, Multiple positive fixed points of nonlinear operators on ordered Banach spaces, {\em Indiana Univ. Math. J.} {\bf28} (1979), 673--688.

\bibitem{li1}: Yongxiang Li, Positive radial solutions for elliptic equations with nonlinear gradient terms in an annulus, {\em Complex Var. Elliptic Equ.}, (2017). DOI:10.1080/17476933.2017.1292261

\bibitem{mei}Linfeng Mei, Structure of positive radial solutions of a quasilinear elliptic problem with singular nonlinearity, {\em Complex Var. Elliptic Equ.}, (2017). DOI:10.1080/17476933.2017.1399367

\bibitem{ni}W. M. Ni, Some aspects of semilinear elliptic equations on $\mathbb{R}^n,$ in Nonlinear Diffusion Equations and Their Equilibrium States II (Ed. by W. M. Ni, L. A. Peletier, J. Serrin), Springer-Verlag, New York (1988), 171--205.

\bibitem{rus}I. A. Rus, On a fixed point theorem of Maia, {\it Studia Univ. "Babes-Bolyai", Mathematica}, {\bf1} (1977), 40--42.


\bibitem{son}B. Son, P. Wang, Positive radial solutions to classes of nonlinear elliptic systems on the exterior of a ball, {\it J. Math. Anal. Appl.} \textbf{488} (2020) 124069.

\bibitem{stin}C. P. Stinson, S. S.  Almuthaybiri, C. C. Tisdell, A note regarding extensions of fixed point theorems involving two metrics via an analysis of iterated functions, {\it ANZIAM J.}, 61(EMAC2019) (2020), C15--C30.

\bibitem{yana}E. Yanagida, Uniqueness of positive radial solutions of $\triangle u+\mathtt{g}(\vert \mathtt{x}\vert)u+\mathtt{g}(\vert \mathtt{x}\vert)u^p=0$ in $\mathbb{R}^n,$ {\it Arch. Rational Mech. Anal.}, {\bf115}(1991), 257--274. 
\end{thebibliography}
\end{document}